\documentclass{article}
\usepackage[utf8]{inputenc}
\usepackage{graphicx}
\usepackage{mathtools,amsmath,amssymb,amsfonts,amsthm}
\usepackage{mathalfa}
\usepackage{tikz}
\usetikzlibrary{calc}
\usepackage{color}
\usepackage{url}

\usepackage{enumitem}

\usepackage{subcaption}

\graphicspath{{./figures/}}

\usepackage[english]{babel}

\usepackage{tabularx}

\usepackage{xcolor}

\usepackage{verbatim}
\usepackage{xspace}

\usepackage{amsfonts}
\usepackage{mathtools}
\usepackage{amsmath}

\usepackage{amssymb}

\usepackage{setspace}

\usepackage{hhline}

\usepackage{fullpage}
\usepackage{ifthen}
\usepackage{tikz}
\usetikzlibrary{automata,positioning, arrows}

\theoremstyle{plain}
\newtheorem{thm}{Theorem}
\newtheorem{lemma}{Lemma}
\newtheorem{prop}{Proposition}

\theoremstyle{definition}

\newtheorem{defn}{Definition}

\theoremstyle{remark}
\newtheorem{example}{Example}
\newtheorem{rmk}{Remark}

\graphicspath{{./figures/}}

 \newcommand{\Pre}{{\rm{Pre}}}
 \newcommand{\Post}{{\rm{Post}}}
 \newcommand{\g}{\textsl{g}}
 \newcommand{\Z}{\mathbb{Z}}
 \newcommand{\cH}{\mathcal{H}}
\newcommand{\wi}{\hat {\imath}}
 \newcommand{\wj}{\hat {\jmath}}
 \newcommand{\wh}{\hat {h}}
 \newcommand{\cL}{\mathcal{L}}
\newcommand{\z}{\overline z}
 \newcommand{\w}{\overline w}

\newcommand{\R}{\mathbb{R}}{}
{}
\newcommand{\N}{\mathbb{N}}

\newcommand{\cA}{\mathcal{A}}
\newcommand{\cB}{\mathcal{B}}
\newcommand{\cC}{\mathcal{C}}
\newcommand{\cD}{\mathcal{D}}

\newcommand{\cK}{\mathcal{K}}
\newcommand{\cKL}{\mathcal{KL}}
\newcommand{\cN}{\mathcal{N}}
\newcommand{\cG}{\mathcal{G}}

\newcommand{\cS}{\mathcal{S}}

\newcommand{\cV}{\mathcal{V}}
\newcommand{\cP}{\mathcal{P}}
\newcommand{\cZ}{\mathcal{Z}}

\newcommand{\wt}{\widetilde}

\DeclarePairedDelimiterX{\inp}[2]{\langle}{\rangle}{#1, #2}

\title{Multiple Lyapunov Functions and Memory: A Symbolic Dynamics Approach to Systems and Control
}
  
\author{Matteo Della Rossa$^*$
  \and  ~Rapha\"el M. Jungers \thanks{R.J. is a FNRS honorary Research Associate. This project has received funding from the European Research Council (ERC) under the European Union's Horizon 2020 research and innovation programme under grant agreement No 864017 - L2C. R.J. is also supported by the Innoviris Foundation and the FNRS (Chist-Era Druid-net). The first and second authors contributed equally. \\The Authors are with the ICTEAM,
        UCLouvain, 4 Av. G. Lema\^{i}tre, 1348 Louvain-la-Neuve, Belgium. ).
        {\small \{matteo.dellarossa,} {\small raphael.jungers\}@uclouvain.be}}}

\begin{document}

\maketitle

\begin{abstract}
We propose a novel framework for the Lyapunov analysis of an important class of hybrid systems, inspired by the theory of symbolic dynamics and earlier results on the restricted class of switched systems. This new framework allows us
to leverage language theory tools in order to provide a universal characterization of Lyapunov stability for
this class of systems. 

We establish, in particular, a formal connection between multiple Lyapunov functions and techniques based on memorization and/or prediction of the discrete part of the state. This allows us to provide an equivalent (single) Lyapunov function, for any given multiple-Lyapunov criterion.

By leveraging our language-theoretic formalism, a new class of stability conditions is then obtained when considering both memory and future values
of the state in a joint fashion, providing new numerical schemes that outperform existing technique.
Our techniques are then illustrated on numerical examples.
\end{abstract}
\section{Introduction}
In recent years, particular attention has been devoted to the analysis of hybrid systems, both from a theoretical and application-oriented point of view.
In this context, the dynamical state exhibits (or is affected by) both continuous-time and discrete-time phenomena, rendering the (stability) analysis challenging, due to the composite nature of the system. For an overview, we refer to~\cite{goebel2012hybrid,Shorten2006}. Among the generic class of hybrid systems, a benchmark example is provided by \emph{switched systems} (for a formal introduction, see~\cite{Lib03,Liberzon99}). This setting can describe, in an equivalent manner, a  class of delay systems~\cite{HetelDaafouzIung08} as well as a large set of piecewise smooth dynamical systems~\cite{Bernardo2008}. Formally, given a finite set of vector fields $\{g_1,\dots, g_M\}\subset \cC(\R^n,\R^n)$, we consider  systems of the form
\begin{equation}\label{eq:SystemIntro}
x(k+1)=g_{\sigma(k)}(x(k)),\;\;\;\;\;\;k\in \N,
\end{equation}
where $\sigma:\N\to \{1,\dots, M\}$ is the \emph{switching signal} describing the discrete behavior, or switching, among the subsystems.
Framework~\eqref{eq:SystemIntro} provides a suitable mathematical model for large families of engineering systems, for example digital circuits~\cite{DeaGer10}, consensus dynamics~\cite{JadLin03}, smart buildings~\cite{ShaikhNoR}, etc.

Due to the presence of the switching signal in~\eqref{eq:SystemIntro}, several connections and relations between this framework and symbolic dynamics/graph theory~(see~\cite{LindMarcus95}) can be established. More formally, the (admissible) switching signals can be interpreted, in a more abstract setting, as elements of a shift-space over a finite alphabet. This approach has been used (although implicitly) to generalize and  abstract switched systems: as a few examples, in~\cite{AazanGir22,PEDJ:16} labeled graphs are used to constrain the set of admissible sequences, while in~\cite{CHITOUR2021101021} a graph-formalism is introduced to provide stability results.
More specifically, graph-theory has been successfully 
exploited in studying Lyapunov stability certificates for~\eqref{eq:SystemIntro}.  Indeed, due to the hybrid nature of the considered class of systems, the search for a \emph{common} Lyapunov function is usually a conservative approach, and graph theory provides a natural tool to generalize this method, via \emph{multiple} Lyapunov certificates (see~\cite{BRa:1998}).
In the seminal papers~\cite{LeeDull06,AJPR:14} it is shown that a graph structure can induce a set of Lyapunov inequalities implying stability of~\eqref{eq:SystemIntro}. For a converse result see~\cite{JunAhm17}, in which it is shown that graph-based Lyapunov conditions can encode \emph{any} valid stability certificate. This literature, studying (various types of) graph-based Lyapunov conditions, has been intensively developed in recent years, see for example~\cite{PEDJ:16,PhiAth19,DebDel22,DelPas22,DonDull20,EssickLee14,AazanGir22}.
In the aforementioned results, summarizing, graph structures are used as a technical tool to provide flexible representations of the admissible switching sequences and/or to encode Lyapunov inequalities. 

In parallel, in the study of general (hybrid) dynamical systems, it is rather common to introduce and model a concept of ``\emph{memory}'', which records the past values of the discrete part of the state, i.e., for~\eqref{eq:SystemIntro}, the past values of the switching signal. As first example, in~\cite{MoorRaish98,YangMoorRaish20, SchmTab15,MajuOzat2020} a concept of uniform memory is used to refine and analyze abstractions of (a class of) non-linear systems, using a finite-state machine approach. Moreover, in the context of behavioral analysis of dynamical systems (see, for example,~\cite{Willems1989,Will07}) the concept of memory or cause/effect abstracting structures is essential. 
As last example, the use of information on past and future values of the switching sequences, in the restricted context of switched systems, is introduced in~\cite{EssickLee14,Thesis:Essick}. 

The strong relations bending graph structures and admissible past/future discrete sequences is well-known in symbolic dynamics and abstract dynamical systems literature. In the seminal paper~\cite{DeBru46}, the so-called de Bruijn graphs were introduced to represent common pasts for sequences of letters in an alphabet. Since then, the bridge between admissible trajectories of systems evolving on a finite state space and admissible paths in a corresponding graph/state-machine representation has become a common topic of symbolic dynamics monographs, as for example~\cite{CassandrasLafortune08,LindMarcus95}. On the other hand, this connection between graph theory and past/future information about the dynamics of a system has not been made explicit in multiple Lyapunov functions theory. As a result, none of the contributions above exploits the full power of symbolic dynamics for building Lyapunov functions, leading to conservatism of the proposed techniques.

% In this manuscript we study and make explicit the aforementioned relation, showing the equivalence (in a sense that we clarify) between graph- and memory/future-based Lyapunov functions criteria. This allows us to: 1) generalize the family of systems on which classical multiple Lyapunov functions can be applied; and 2) provide new numerical schemes, which we show can be more powerful in practice.

In this manuscript we study and make explicit the aforementioned relation, showing the equivalence (in a sense that we clarify) between graph- and memory/future-based Lyapunov functions criteria. We introduce a general hybrid systems model, encompassing classical settings, as, for example, time-varying systems, switched systems, systems depending on past values of discrete signals, etc. This framework allows us to provide novel converse Lyapunov results, generalizing existing stability characterizations. 
We then specialize these results to the case of finite-representation of languages/sequences, providing novel techniques in the field of graph-based Lyapunov conditions. In this context, we provide novel duality results, exploiting and clarifying the connections between past/future horizon, convex conjugation of functions and duality of linear systems. The numerical benefits of the introduced tools are then illustrated considering the stability analysis of linear switched systems.
\\
\noindent
\textbf{Outline:}
We first present preliminaries from graph- and language- theory in Section~\ref{sec:Prelim}.
Then, in Section~\ref{sec:MainSections}, starting from the partial analysis performed in the preliminary papers~\cite{DelRosJun23b,DelRosJun23a}, we consider a more abstract and more general formulation of systems as~\eqref{eq:SystemIntro}; modeling the (admissible) switching signals as additional states lying on a shift-space. This allows us to provide a general theory of \emph{sequence-dependent Lyapunov functions}, showing through a converse theorem that this formalism characterizes the stability property. The introduced framework also permits to consider, in a joint manner, conditions based on \emph{past and future} information, thus bypassing the classical dichotomy between memorization (of past events) and prediction (of future events), and thus generalizing the results presented in~\cite{LeeDull06,Thesis:Essick,DelRosJun23b,DelRosJun23a}. 

In Section~\ref{sec:GraphvsPart}, we then consider \emph{finite} coverings of shift spaces induced by graphs allowing us to provide algorithmically actionable Lyapunov criteria induced by coverings. We perform an in-depth analysis of the introduced formalism, connecting properties of graphs with properties of the corresponding coverings and the arising Lyapunov conditions.
Our approach not only formalizes and generalizes the above-mentioned classical techniques, but also  also opens an avenue for a systematic controller design technique for arbitrary dynamical systems, with the help of finite state machines, with promising applications in cyber-physical systems. In Section~\ref{sec:LinearAndNumerical}, we present some numerical examples, showing that the introduced framework also allows for a drastic reduction in computational complexity, which may be of an exponential factor.

\paragraph{Notation} The set of natural numbers is defined by $\N:=\{0,1,\dots\}$, while the set of negative integers is defined by $\Z_-:=\{-1,-2,\dots\}$. The set of non-negative reals is defined by $\R_{\geq 0}:=[0,+\infty)$. The set $\mathbb{S}^n_{+}$ is the set of positive definite matrices in $\R^{n\times n}$. Given a set $A$, $\cP(A)$ denotes its power set. A function $\alpha:\R_{\geq 0}\to \R_{\geq 0}$ is of \emph{class $\cK$} ($\alpha \in \cK$) if $\alpha(0)=0$, it is continuous and strictly increasing; it is of \emph{class $\cK_\infty$}($\alpha \in \cK_\infty$) if, in addition, it is unbounded. A continuous function $\beta:\R_+\times \R_+\to \R_+$ is of \emph{class $\mathcal{KL}$} if $\beta(\cdot,s)$ is of class $\cK$ for all $s$, and $\beta(r,\cdot)$ is decreasing and $\beta(r,s)\to 0$ as $s\to\infty$, for all $r$.

\section{Preliminaries}\label{sec:Prelim}
\subsection{Shift Spaces and Language Theory}
We provide here a concise review of the necessary concepts from language/symbolic dynamics theory. For an exhaustive overview of this topic, we refer to~\cite{LindMarcus95}.

Consider a countable set $\Sigma$, also called the \emph{alphabet}. Given any $K\in \N$ we denote by $\Sigma^K$ the $K$ cartesian product of $\Sigma$, i.e., $\Sigma^K:=\{(i_0,\dots, i_{K-1})\;\vert\; i_k\in\Sigma,\forall k\leq K\}$.
We have the following definition.
\begin{defn}
The \emph{full-$\Sigma$ shift space}, denoted by $\Sigma^\Z$, is the set of all the bi-infinite sequences of elements of $\Sigma$. Formally
\[
\Sigma^\Z:=\{\z=(z_i)_{i\in \Z}\;\,\vert\;\, z_k\in \Sigma,\;\forall k\in \Z\}.
\]
On the set $\Sigma^\Z$ we define the \emph{shift function} $\sigma:\Sigma^\Z\to \Sigma^\Z$ by
\[
\sigma:\z\mapsto \overline w,\text{ with } w_k=z_{k+1}.
\] 
\end{defn}
In the following we introduce some notation used in what follows.
\begin{defn} 
\begin{itemize}[leftmargin=*]  \setlength\itemsep{0em}
\item We define the \emph{past} and \emph{future} (one-sided) sequences of $\Sigma$ by
\[
\begin{aligned}
\Sigma^-&:=\{(\dots,z_{-2},z_{-1})\;\vert\; z_k\in \Sigma, \forall k \in \Z_-\},\\
\Sigma^+&:=\{(z_0,z_1,\dots)\;\vert\; z_k\in \Sigma, \forall k \in \N\}
\end{aligned}
\]
\item Given $\z=(\dots,z_{-1},z_0,z_1,\dots)\in \Sigma^\Z$, we write $\z=\z^-\cdot \z^+$ with $\z^-:=(\dots, z_{-1})\in \Sigma^-$ the \emph{past sequence} of $\z$ and $\z^+:=(z_0,z_1,\dots)\in \Sigma^+$ the \emph{future sequence} of $\z$. 
\item 
For any set $A\subseteq \Sigma^\Z$, we define $\Pre(A):=\{\z^-\in \Sigma^-\;\vert\;\z\in A\}$ and $\Post(A):=\{\z^+ \in \Sigma^+\;\vert\;\z\in A\}$.
\item A \emph{word} of $\Sigma$ is a finite sequence of symbols from $\Sigma$, and we use the notation $\wi:=(i_0,\dots, i_{k-1})\in\Sigma^*:=\bigcup_{k\in \N}\Sigma^k$ (the \emph{Kleene closure of $\Sigma$}). Given words $\wi,\wj\in \Sigma^\star$, $\wi\cdot\wj\in \Sigma^\star$ (or, for simplicity, $\wi\wj$) denotes the concatenation of $\wi$ and $\wj$. We denote by $|\wi|$ the \emph{length} of the word $\wi$, i.e. $|\wi|=k$ if $\wi \in \Sigma^k$.
  \item
Given a word $\wi \in \Sigma^k$ of length $k\neq 0$, and given $a,b\in \Z$ such that $b-a+1=k$ we define the set\footnote{ The set $[\wi]_{[a,b]}$ is sometimes called a ``cylinder'' associated to $\wi$, for example in~\cite[Section 6.1]{LindMarcus95}.}
\[
[\wi]_{[a,b]}:=\{\overline w\in \Sigma^\Z\;\vert\;\;\;w_a=i_0,\dots,w_b=i_{k-1}\}.
\]

\item Consider the \emph{time-inversion function} $\eta:\Z\to\Z$ defined by $\eta(k)=-1-k$. Given $\z=(\dots,j_{-1},j_{0},j_{1})\in \Sigma^\Z$, we denote the \emph{time-inverse} of $\z$, by $\z^{-1}=(\dots, h_{-1},h_0,h_1\dots)$ with $h_{k}=j_{\eta(k)}$ for every $k\in \Z$.
\end{itemize}
\end{defn}

\subsection{Graph-Theory}
In this subsection we recall the necessary definitions from graph theory.
\begin{defn}[(Labeled) Graphs]
Given an alphabet $\Sigma$, a (labeled) graph $\cG$ on $\Sigma$ is defined by $\cG=(S,E)$, where $S$ is a finite \emph{set of nodes} and $E\subseteq S\times S\times \Sigma$ is the \emph{set of (labeled) edges}.
Given $e=(s,q,i)\in E$, $s$ and $q$ are the \emph{starting} and \emph{arrival} node of $e$, respectively, while, the projection function $\ell:E\to \Sigma$ is called the \emph{labeling} function. Two edges $e=(s_1,q_1,i_1)\in E$ and $f=(s_2,q_2,i_2)\in E$ are said to be \emph{consecutive} if $q_1=s_2$.
\end{defn}
\begin{defn}[Paths and Infinite Walks]
Given $\cG=(S,E)$ and a word $\wi=(i_0,\dots, i_{K-1}) \in \Sigma^K$, a \emph{path} on $\cG$ labeled by $\wi$  is a sequence of consecutive edges $\overline e=e_1,\dots,e_K=(s_0,s_1,i_0),(s_1,s_2,i_1)\dots, (s_{K-1},s_{K},i_{K-1})\in E^K$ labeled by $\wi$; $s_0$ and $s_K$ are the starting and arrival nodes of the path, respectively.\\ Given $\z=(\dots, z_{-1},z_0,z_1,\dots)\in \Sigma^\Z$ a \emph{(bi-)infinite walk} labeled by $\z$ is a bi-infinite sequence of consecutive edges $\overline \pi=(\dots,e_{-1},e_{0},e_{1}\dots)\in E^\Z$, such that $\ell(e_k)=z_k$ for all $k\in \Z$. Given a (bi-)infinite walk $\overline \pi=(\dots,e_{-1},e_{0},e_{1}\dots)$ we say that $s\in S$ is its \emph{initial node} if $e_1=(s,q,i)$ for some $q\in S$ and $i\in \Sigma$. In a similar manner, we define \emph{backward and forward one-sided infinite walks}.
\end{defn}
We also introduce the following notation.
\begin{defn}
Given a graph $\cG=(S,E)$ on $\Sigma$, we say that:
\begin{itemize}[leftmargin=*]  \setlength\itemsep{0em}
\item $\cG$ is \emph{strongly connected} if for any $s,q\in S$ there exists a path $\overline e$ starting at $s$ and arriving at $q$;
\item $\cG$ is \emph{deterministic}, if, for any $s\in S$, and any $i\in \Sigma$ there is \emph{at most} one $q\in S$ such that $(s,q,i)\in E$;
\item $\cG$ is \emph{complete}\footnote{\textcolor{black}{The introduced notion of completeness arises from automata theory, and should not be confused with the classical notion of  completeness of graphs requiring the existence of any possible edge, i.e. $E=S\times S\times \Sigma$.}} if, for any $s\in S$, and any $i\in \Sigma$ there is \emph{at least} one $q\in S$ such that $(s,q,i)\in E$;
\item $\cG^\top=(S^\top,E^ \top)$ defined by $S^\top \equiv S$, and $(s,q,i)\in E \,\Leftrightarrow\,(q,s,i)\in E^\top$, is the \emph{transpose graph} of $\cG$, i.e. the graph obtained by reversing the direction of the edges of $\cG$;
\item $\cG$ is \emph{co-deterministic}, if, for any $q\in S$, and any $i\in \Sigma$ there is \emph{at most} one $s\in S$ such that $(s,q,i)\in E$; or, equivalently, if $\cG^\top$ is deterministic;
\item $\cG$ is \emph{co-complete} if, for any $q\in S$, and any $i\in \Sigma$ there is \emph{at least} one $s\in S$ such that $(s,q,i)\in E$;
or, equivalently, if $\cG^\top$ is complete.
\end{itemize} 

\end{defn}

\subsection{Sofic Shifts and Graph Presentations}
The content and notation of this subsection, in which we introduce and study a sub-class of subsets of the full shift $\Sigma^\Z$, is borrowed from~\cite{LindMarcus95}.

\noindent 
Given a labeled graph $\cG=(S,E)$ we define $\cZ(\cG)\subseteq \Sigma^\Z$ by
\begin{equation}\label{eq:LanguageGraph}
\cZ(\cG)=\left \{\z\in \Sigma^\Z\;\vert\;\;\exists \text{ bi-infinite walk } \overline \pi \text{ in } \cG \text{ labeled by } \z\right \}.\tag{1a}
\end{equation}
Moreover given any $\cG=(S,E)$ on $\Sigma$ and any $s\in S$, we define
\begin{equation}\label{eq:LanguagePath-Node}
\cZ(\cG,s)=\left \{\z\in \Sigma^\Z\;\vert\;\;\exists \text{ bi-infinite walk } \overline \pi \text{ in } \cG \text{ labeled by } \z\, \text{ starting at } s\right \}.\tag{1b}
\end{equation}
\refstepcounter{equation}
\begin{defn}[Sofic Shifts]
A set $Z\subseteq\Sigma^\Z$ is a \emph{sofic shift} if there exists a graph $\cG$ on $\Sigma$ such that
\[
Z=\cZ(\cG).
\]
In this case we say that $\cG$ is a \emph{presentation of} $Z$.
\end{defn}

For any non-empty sofic shift there is an infinite number of possible presentations.
It is important to note that any sofic shift is uniquely determined by its finite sub-sequences, as we define in what follows.
\begin{defn}[Language generated by a Sofic Shift]
Given a sofic shift $Z\subseteq\Sigma^\Z$, let us define $\cL(Z)\subseteq \Sigma^*$ as the set of all the possible sub-words of elements in $Z$, i.e.
\[
\wi \in \Sigma^K\cap\cL(Z)\;\Leftrightarrow\; [\wi]_{[0,K-1]}\cap Z\neq \emptyset.
\] 
\end{defn}
We have the following characterization and properties of languages generated by a sofic shift.
\begin{lemma}
A language $\cL\subseteq \Sigma^\star$ is generated by a sofic shift $Z\subseteq \Sigma^\Z$ if and only if it is a \emph{regular language}\footnote{We refer to~\cite{CassandrasLafortune08} for the formal definition. Intuitively, a language is regular if it is accepted by a deterministic finite automaton.}. In particular, given a sofic shift $Z$ and $\cL(Z)$, we have the following properties:
\begin{itemize}[leftmargin=*]  \setlength\itemsep{0em}
\item For any $\wi \in \cL(Z)$ and any sub-word $\wj$ of $\wi$, we have $\wj\in \cL(Z)$;
\item For any $\wi\in \cL(Z)$ there exist non-empty words $\wj,\wh\in \cL(Z)$ such that $\wj\cdot\wi\cdot\wh\in \cL(Z)$.
\end{itemize} 
\end{lemma}
We have the following important equivalence result.
\begin{lemma}\label{Lemma:EquivalenceLanguage}
Given sofic shifts $Z_1,Z_2\subseteq \Sigma^\Z$, we have $\cL(Z_1)=\cL(Z_2)$ if and only if $Z_1=Z_2$, i.e. a sofic shift is uniquely determined by its generated language.
\end{lemma}
For the proof we refer to~\cite[Proposition 1.3.4]{LindMarcus95}. As an example, we have that $\cL(\Sigma^\Z)=\Sigma^*$.
\begin{defn}\label{defn:Irreducible}
A sofic shift $Z\subseteq \Sigma^\Z$ is \emph{irreducible} if, for every $\wi,\wj\in \cL(Z)$ there exists a $\wh\,\in \cL(Z)$ such that $\wi\,\wh\wj\in \cL(Z)$, or equivalently, if it has a strongly connected presentation (see \cite[Proposition 3.3.11]{LindMarcus95}).
\end{defn}

\section{Dynamical Systems and Sequence-dependent Lyapunov Functions}\label{sec:MainSections}
\subsection{System Definition and General Converse Lyapunov Theorems}
In this section we introduce dynamical systems that \emph{jointly} evolve on a continuous state-space ($\R^n$ for some $n$) and on a sofic shift. This family includes several classical models in systems and control, such as switched systems, or time-varying systems. We define the considered notion of stability and we provide the corresponding Lyapunov characterization.
\begin{defn}\label{defn:System}
 Given any sofic shift $Z\subseteq \Sigma^\Z$ and a function $f:\R^n\times Z\to \R^n$  we study dynamical systems evolving on $\R^n\times Z$, defined as follows:
\begin{equation}\label{eq:System}
\begin{cases}
x(0)=x_0\in \R^n,\\
{}\omega(0)=\overline z\in Z,\\
x(k+1)=f(x(k),\omega(k)),\\
\omega(k+1)=\sigma(\omega(k)).
\end{cases}
\end{equation}
If there exists $\wt f:\R^n\times \Sigma\to \R^n$ such that
\[
f(x,\z)=\wt f(x,z_0),\;\;\;\forall \;\z\in Z,
\] 
i.e. if the vector field only depends on the $0$-position element of any point in $\z\in Z$ (i.e. the ``current value'' of $\z$), then the system is said to be a \emph{switched system}. 
\end{defn}
We denote by $\Phi(k,x_0,\z)\in \R^n$ the \emph{state-solution} of \eqref{eq:System} evaluated at time $k\in \N$. Note that the state-solution of~\eqref{eq:System} satisfies the following semigroup property: for any $x_0\in \R^n$, any $0\leq h\leq k\in \N$ and any $\z\in Z$, we have
\begin{equation}\label{eq:SemiGroupProperty}
\Phi(k,x_0,\z)=\Phi(k-h,\Phi(h,x_0,\z),\sigma^h(\z)).
\end{equation}
Moreover, we also define the backward state-solution set, for negative $k$, by
\begin{equation}\label{eq:BackwardSolution}
\Phi(k,x_0,\z):=\{y\in \R^n\;\vert\; \Phi(-k, y, \sigma^{k}(\z))=x_0\}, \;\;\forall k\in \Z_{-}.
\end{equation}
Note that if $f(\cdot,\z):\R^n\to \R^n$ is invertible for any $\z \in Z$, then $\Phi(k,x_0,\z)$ is a singleton, for any $\z\in Z$, any $x_0\in \R^n$ and any $k\in \Z_-$.

\begin{rmk}[Systems Class and Related Models]\label{remark:ModelSystems}
The framework~introduced in Definition~\ref{defn:System} provides a rather general model. 
First of all, for the switched systems case, we recover the systems class~\eqref{eq:SystemIntro} discussed in the Introduction, simply defining $g_i(x):=f(x,i)$. The sofic shift $Z$  encodes, in this case, the constraints on the feasible switching sequences. We thus recover the framework studied in~\cite{PEDJ:16}, and, as a by product, also the class of discrete-time delay systems, which can be rewritten, via a state augmentation technique, in the form~\eqref{eq:System} (cft.~\cite{HetelDaafouzIung08} and references therein).

On the other hand, Definition~\ref{defn:System} generalizes the framework of switched systems, since the vector field can depend, in the non-switched case, not only on the current ``mode'' but also on previous/future values (since in general, it depends on bi-infinite sequences). This is notably the case when one considers the \emph{stabilization} problem for switched systems in terms of feedback maps depending on the past/future values of the switching signals, as for example in~\cite{EssickLee14,LeeDull06,DonDull20} and references therein.
Moreover, the setting of \emph{time-varying systems} can also be seen as a special case of Definition~\ref{defn:System}. Indeed, consider the alphabet $\Sigma=\{\circ,\bullet\}$ and the sofic shift $Z_\Z$ generated by the graph $\cG_\Z$ in Figure~\ref{fig:IntegerGraph}. It is clear that $Z_\Z$ is the set of bi-infinite sequences with at most one occurrence of the symbol ``$\circ$''.  Let us call $\z_\infty\in Z_\Z$ the bi-infinite sequence with \emph{no} occurrence of $\circ$. There exists a bijection $\Theta:Z_\Z\to \Z\cup\{\infty\}$ defined by 
\[
\begin{aligned}
&\Theta(\overline z_\infty):=\infty, \text{ otherwise}\\
&\Theta(\overline z):=1-k, \text{ for the unique }k\in \Z\;\text{ such that } \;z_k=\circ.
\end{aligned}
\]
It can be seen that we have $\Theta(\sigma(\z))=\Theta(\z)+1$, (with the convention that $\infty+1=\infty$) and thus the shift on $Z_\Z$ corresponds to the usual time-shift in $\Z$.
Due to this bijection, given any $f:\Z\times \R^n\to \R^n$, the time-varying system of the form
\[
x(k+1)=f(k,x(k)),
\]
can be seen and studied as a dynamical system on $\R^n\times Z_\Z$, in the sense of Definition~\ref{defn:System} (by defining a trivial dynamics $f(\cdot,\z_\infty)\equiv 0$).
For this reason, some of the proofs presented in what follows are partially inspired by and can be seen as generalization of results in the context of time-varying systems. In particular, our subsequent Theorems~\ref{lemma:ConverseNonLinear} and~\ref{Lemma:ConverseLinear} will recover and generalize classical statements for \emph{uniform} asymptotic stability of time-varying systems, as the ones in the seminal~\cite{Hahn1958}, see also~\cite{IchKat01,JiaWan02,KelTel04}.
\end{rmk}

\begin{figure}[t!] 
  \centering
  \vspace{0.0cm}
    \begin{tikzpicture}    [>=stealth,
    shorten >=1pt,
    node distance=2.5cm,
    on grid,
    auto,
    every state/.style={draw=red!60, fill=red!5, thick}
    ]
\node[state,inner sep=1pt, minimum size=22pt] (left)                  {$p$};
\node[state,inner sep=1pt, minimum size=22pt] (right) [right=of left] {$f$};
\path[->]
%   FROM       BEND/LOOP           POSITION OF LABEL   LABEL   TO
   (left) edge[bend left=10]     node          [scale=1]            {$\circ$} (right)
   (left) edge[loop left=60]     node            [scale=1]          {$\bullet$} (left)
   (right) edge[loop right=60]     node      [scale=1]                {$\bullet$} (right)
   ;
\end{tikzpicture}
\caption{The graph $\cG_\Z$, presentation of the shift $Z_\Z$, defined in Remark~\ref{remark:ModelSystems}.}\label{fig:IntegerGraph}
  \end{figure}
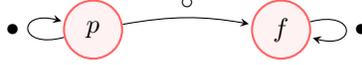

Since we are interested in stability of~\eqref{eq:System} with respect to a point in $\R^n$ (w.l.o.g., the origin), we introduce the following definition.

\begin{defn}[Global Uniform Asymptotic Stability]\label{defn:GUAS}
Given any sofic shift $Z\subseteq \Sigma^\Z$ and any $f:\R^n\times Z\to \R^n$, system~\eqref{eq:System} is said to be \emph{globally uniformly  asymptotically stable} (GUAS) if there exists $\beta\in \cKL$ such that
\[
|\Phi(k,x_0,\z)|\leq \beta(|x_0|,k),\;\;\;\forall k \in \N,\;\forall x_0\in \R^n, \;\forall \z\in Z.
\]
It is said \emph{uniformly exponentially stable} (UES) if there exist $M>0$ and $\gamma\in [0,1)$ such that
\begin{equation}\label{eq:ExponentialStabilityIneq}
|\Phi(k,x_0,\z)|\leq M \gamma^k|x_0|\;\;\forall k \in \N,\;\forall x_0\in \R^n, \;\forall \z\in Z.
\end{equation}
In this case the scalar $\gamma$ is called a \emph{decay rate} of the system.
\end{defn}
We now state the Lyapunov characterization of the GUAS for sequence-dependent dynamical systems.
\begin{defn}[Sequence-Dependent Lyapunov functions]\label{defn:SequenceDepLyapFunct}
A function $\cV:\R^n\times Z\to \R$ is a \emph{sequence-dependent Lyapunov function} \emph{(sd-LF)} for system~\eqref{eq:System} if  there exist $\alpha_1,\alpha_2\in \cK_\infty$, and $\gamma\in [0,1)$ such that
\begin{subequations}
\begin{equation}\label{Eq:Sandwich1}
\alpha_1(|x|)\leq \cV(x,\z)\leq \alpha_2(|x|),\;\;\forall x\in \R^n,\;\forall\, \z\in Z,
{}\end{equation}
\begin{equation}\label{eq:Decreasing1}
\cV(f(x,\z),\sigma(\z))\leq \gamma \cV(x,\z),\;\;\forall x\in \R^n,\;\forall\, \z\in Z.
\end{equation}
\end{subequations}
\end{defn}

\begin{thm}[Converse Lyapunov Result: Non-linear case]\label{lemma:ConverseNonLinear}
Consider any sofic shift $Z\subseteq \Sigma^\Z$ and any $f:\R^n\times Z\to \R^n$.
System~\eqref{eq:System} is GUAS (in the sense of Definition~\ref{defn:GUAS}) if and only if there exists a sd-LF for system~\eqref{eq:System}.
\end{thm}
Theorem~\ref{lemma:ConverseNonLinear} provides a complete characterization of GUAS of systems as in~\eqref{eq:System} in terms of Lyapunov functions, as in the classical discrete-time non-linear systems case (see for example~\cite{KelTel04}). In other words, we have proven that the notion of sequence-dependent Lyapunov function introduced in Definition~\ref{defn:SequenceDepLyapFunct} is the most appropriate in studying stability of~\eqref{eq:System}. 
Although Theorem~\ref{lemma:ConverseNonLinear} is new, its proof substantially follows classical arguments in Lyapunov theory; we present a complete proof in Appendix~\ref{appendix:ProofConverseNonLinear}. Recalling the discussion in Remark~\ref{remark:ModelSystems}, we note that in the time-varying systems case, we recover, as particular case, the converse result in~\cite[Lemma 2.7]{JiaWan02} (in the case with zero disturbance).
\noindent
\paragraph{Linear case:}
When the function $f:\R^n \times Z \to \R^n$ is linear in $x$, i.e. in the case $f(x,\z)=A(\z)x$ with $A:Z\to \R^{n\times n}$ a matrix-valued map, we call~\eqref{eq:System} a \emph{(sequence-dependent) linear dynamical system}. From a classical homogeneity argument, it can be seen that, for linear dynamical systems, GUAS is equivalent to UES, see for example~\cite{Rosier92,BacRosier}. Moreover, we can refine Theorem~\ref{lemma:ConverseNonLinear} considering \emph{quadratic} signal-dependent Lyapunov functions, as stated below.
\begin{thm}[Converse Lyapunov Result: Linear case]\label{Lemma:ConverseLinear}
Consider a sofic shift $Z\subseteq \Sigma^\Z$, let us fix $\gamma\in [0,1)$. The following statements are equivalent:
\begin{enumerate}[leftmargin=0.65cm]
\item[$(1)$] For any $\wt \gamma\in (\gamma,1)$ the linear system~\eqref{eq:System} with $f(x,\z)=A(\z)x$ with $A:Z\to \R^{n\times n}$ is UES (on $Z$) with decay rate $\wt \gamma$;
\item[$(2)$] For any $\wt \gamma\in (\gamma,1)$, there exist $M_1,M_2>0$ and $Q:Z\to \mathbb{S}^n_{+}$ such that
\begin{subequations}
\begin{equation}\label{eq:SandwichLinear}
 M_1I_n\preceq Q(\z)\preceq M_2 I_n,\;\;\;\forall\;\z\in Z,
\end{equation}
\begin{equation}\label{eq:PseudoRiccatti}
A(\z)^\top Q(\sigma(\z))A(\z)\prec \wt \gamma^2\, Q(\z)\;\;\;\forall\;\z \in Z.
\end{equation}
\end{subequations}
\end{enumerate}
\end{thm}
The proof is reported in~Appendix~\ref{Appendix:ProofconverseLinear}. In the time-varying linear systems case, we recover the classical result (see for example~\cite[Section 1.5]{HalaIonescu94} or~\cite[Proposition 3.2]{IchKat01}) establishing the equivalence of uniform exponential stability and the existence of a time-varying quadratic Lyapunov function.
\begin{rmk}
As we show in the following section, the introduced framework allows to generalize the multiple Lyapunov functions approach to more general systems than switched systems. 
In fact, our abstract approach describes and structures the Lyapunov conditions in terms of the ``manifest behavior'' (in Willems' terminology, see, e.g.,~\cite{Willems1989,Will07,TotWil11}) of the dynamical system, not in terms of the specific switching signal. By doing so, we can naturally generalize multiple Lyapunov functions criteria to any system whose manifest behavior can be modeled or at least over-approximated by a sofic shift. As we lay here the theoretical foundations of the approach, we leave for further work this abstract view of it, however we believe that it could lead to even further generalizations.
\end{rmk}

\section{Graphical presentation of sofic shifts: a bridge to algorithmic Lyapunov Theory}\label{sec:GraphvsPart}
In this section, we develop formal tools relating graph presentations, coverings of sofic shifts and Lyapunov-based stability criteria. The main idea is to provide finite coverings of a sofic shift, in order to turn equations~\eqref{Eq:Sandwich1}-\eqref{eq:Decreasing1} (or~\eqref{eq:SandwichLinear}-\eqref{eq:PseudoRiccatti} in the linear case) into finitely verifiable criteria. Then, combinatorial considerations on the coverings will allow us to understand properties of the corresponding stability criteria.

\subsection{Graphs and Finite Coverings of Sofic Shifts}
\label{subsection:GraphAndPartion}
 We are interested in coverings of sofic shifts induced by particular graph presentations, as introduced in what follows.
\begin{defn}[Graph-Induced Coverings]\label{Defn:finiteParti}
Consider a sofic shift $Z\subset \Sigma^\Z$ and a set $\cC=\{C_1,\dots, C_K\}\subset \cP(Z)$. $\cC$ is said to be a \emph{graph-induced covering} ($\g$-covering) of $Z$ if there exists a graph $\cG=(S,E)$ on $\Sigma$, with $S=\{s_1,\dots, s_K\}$ such that 
\[
\begin{aligned}
Z&=\cZ(\cG),\\
C_j&=\cZ(\cG,s_j),\;\;\forall \;j\in \{1,\dots K\}.
\end{aligned}
\]
In this case, $\cG$ is said to be a \emph{presentation} of $\cC$.
If it holds that $C_i\cap C_j=\emptyset$, for all $i\neq j$, then $\cC$ is said to be \emph{not-redundant}. 
\end{defn}
The idea of $\g$-covering is inspired by the definition and theory of graph presentations of sofic shifts~\cite[Chapter 3]{LindMarcus95}. However, in our contribution, graphs are not only a representation tool for sofic shifts, since we are also interested in the properties of the arising coverings and, in subsequent sections, to the arising Lyapunov conditions.
In the following we state some important properties of $\g$-coverings, pertaining  to the language-theory interpretation of graphs.
\begin{prop}[Properties of $\g$-coverings]\label{prop:PropertiesOFFP}
Consider a sofic shift $Z\subset \Sigma^\Z$ and a $\g$-covering $\cC=\{C_1,\dots, C_K\}\subset \cP(Z)$ and suppose $\cG=(S,E)$ is a graph presentation of $\cC$, with $S=\{s_C\}_{C\in \cC}$.
Then we have 
\begin{align}\label{eq:Covering}
\bigcup_{C\in \cC}C&=Z.
\end{align}
Moreover, for all $i\in \Sigma$, consider the set-valued function $\cS_i:\cC\to \cP(\cC)$ defined by
\[
\cS_i(C):=\{D\in \cC\;\vert\;(s_C,s_D,i)\in E\}.
\]
 We have that
\begin{subequations}
\begin{equation}\label{eq:Propagation1}
{\rm{Post}}(C\cap\,[i]_{[0,0]})=i\cdot\bigcup_{D\in\cS_i(C)}{\rm{Post}}(D),
\end{equation}
\begin{equation}\label{eq:Propagation2}
{\rm{Pre}}(D\cap [i]_{[-1,-1]})=\big (\bigcup_{C:\;D\in\cS_i(C)}{\rm{Pre}}(C)\,\big)\cdot i.
\end{equation}
\end{subequations}
In particular, for any $C\in \cC$, any $i\in \Sigma$ and any $D\in \cS_i(C)$ we have
\begin{subequations}
\begin{equation}\label{eq:SuffPropagation1}
i\cdot {\rm{Post}}(D)\subseteq {\rm{Post}}(C),
\end{equation}
\begin{equation}\label{eq:SuffPropagation2}
{\rm{Pre}}(C)\cdot i\subseteq {\rm{Pre}}(D).
\end{equation}
\end{subequations}
\end{prop}

\begin{proof}
Consider a sofic shift $Z$ and a $\g$-covering $\cC$ induced by a graph $\cG=(S,E)$.
Condition~\eqref{eq:Covering} follows from Definition~\ref{Defn:finiteParti} and noticing that 
\[
Z=\cZ(\cG)=\bigcup_{s\in S}\cZ(\cG,s)=\bigcup_{C\in \cC}C.
\]
Then, consider any $C\in \cC$, any $i\in \Sigma$ and any $\z\in \Sigma^\Z$, we have
\[
\begin{aligned}
\z^+&\in \text{Post}([i]_{[0,0]}\cap C)\;\Leftrightarrow\\&\exists\; s_D\in S\;\text{such that } \left(\begin{aligned} (s_C,s_D,i)\in E\; \wedge \exists &\text{\;one-sided infinite path in $\cG$}\\&\text{  labeled by }\sigma(\z)^+\text{ starting at }D\end{aligned}\right)\;\Leftrightarrow\\&
\z^+\in i\cdot \bigcup_{D\in\cS_i(C)}{\rm{Post}}(D),
\end{aligned}
\]
proving~\eqref{eq:Propagation1}. Property~\eqref{eq:Propagation2} can be proven with similar steps.

Consider any $C\in \cC$, any $i\in \Sigma$ and any $D\in \cS_i(C)$. Conditions~\eqref{eq:SuffPropagation1} trivially follows by~\eqref{eq:Propagation1}. 
Then, since $D\in \cS_i(C)$ and thus, by~\eqref{eq:Propagation2}, $\Pre(C)\cdot i\subseteq \Pre(D)$.
\end{proof}

In what follows, we provide a first example of $\g$-covering, and we illustrate how Proposition~\ref{prop:PropertiesOFFP} can be used to state if a given covering of a sofic shift $Z$ is  not a $\g$-covering.
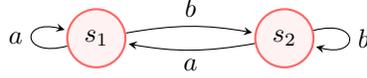
\begin{figure}[t!] 
  \centering
  \vspace{0.0cm}
    \begin{tikzpicture}    [>=stealth,
    shorten >=1pt,
    node distance=2.5cm,
    on grid,
    auto,
    every state/.style={draw=red!60, fill=red!5, thick}
    ]
\node[state,inner sep=1pt, minimum size=22pt] (left)                  {$s_1$};
\node[state,inner sep=1pt, minimum size=22pt] (right) [right=of left] {$s_2$};
\path[->]
%   FROM       BEND/LOOP           POSITION OF LABEL   LABEL   TO
   (left) edge[bend left=10]     node          [scale=1]            {$b$} (right)
        (right)   edge[bend left=10] node        [scale=1]        {$a$} (left)
   (left) edge[loop left=60]     node            [scale=1]          {$a$} (left)
   (right) edge[loop right=60]     node      [scale=1]                {$b$} (right)
   ;
\end{tikzpicture}
\caption{The graph $\cG_1$, corresponding to the covering $\cC_1=\{[a]_{[-1,-1]},[b]_{[-1,-1]}\}$ in Example~\ref{EX:FirstExample}.}\label{fig:FirstDeBrujin}
  \end{figure}

\begin{example}\label{EX:FirstExample}
We present a first simple example of $\g$-covering.
Consider the graph $\cG_1=(S_1,E_1)$ on the alphabet $\Sigma=\{a,b\}$ with $S=\{s_1,s_2\}$ and $E=\{(s_1,s_1,a),(s_1,s_2,b), (s_2,s_2,b), (s_2,s_1,a)\}$ represented in Figure~\ref{fig:FirstDeBrujin}.
The arising $\g$-covering $\cC_1=\{B_1,B_2\}$ of $\Sigma^\Z$ is given by $B_1=[a]_{[-1,-1]}$, and $B_2=[b]_{[-1,-1]}$. This covering satisfies the conditions of Proposition~\ref{prop:PropertiesOFFP}, by defining $\cS_a(B)=B_1$ for all $B\in \cB$ and $\cS_b(B)=B_2$ for all $B\in \cB$. First, since $\Post(B_1)=\Post(B_2)=\Sigma^+$,~\eqref{eq:Propagation1} is satisfied.
For condition~\eqref{eq:Propagation2}, observe that $\Pre(B_1)=\Pre([a]_{[-1,-1]})$ and $\Pre(B_2)=\Pre([b]_{[-1,-1]})$, implying  
\[
\begin{aligned}
(\Pre(B_1)\cup\Pre(B_2))\cdot a=\Pre(B_1)\;\; \text{ and }\;\;(\Pre(B_1)\cup\Pre(B_2))\cdot b=\Pre(B_2), 
\end{aligned}
\]
concluding the discussion.
\end{example}

\begin{example}
Conditions~\eqref{eq:Propagation1}~\eqref{eq:Propagation2} in Proposition~\ref{prop:PropertiesOFFP} can be seen as necessary conditions for being a $\g$-covering, as introduced in Definition~\ref{Defn:finiteParti}, and thus used to prove that a given covering of a sofic shift is \emph{not} a $\g$-covering.
Consider, as an example, the alphabet $\Sigma=\{a,b\}$, the full shift $\Sigma^\Z$ and the partition $\cC=\{C_1,C_2\}$ defined by $C_1:=[a]_{[-2,-2]}$ and $C_2:=[b]_{[-2,-2]}$. Intuitively, $\cC$ is composed by two classes of bi-infinite sequences which take value $a$ or $b$, respectively, at instant of time $-2$. 
We see that $\cC$ is not a graph-induced covering, since it cannot satisfy the conditions of Proposition~\ref{prop:PropertiesOFFP}.
\noindent
 We first note that $\Sigma^\Z=C_1\cup C_2$ and $C_1\cap C_2=\emptyset$. Then, suppose by contradiction that functions $\cS_i:\cC\to 2^\cC$ as in Proposition~\ref{prop:PropertiesOFFP} (i.e. satisfying conditions~\eqref{eq:Propagation1}~\eqref{eq:Propagation2}) exist. We note that condition~\eqref{eq:Propagation2} is not satisfied by $C_1$ nor $C_2$: consider $\omega_1,\omega_2\in \Pre(C_1)$ defined by $\omega_1\in \Pre([a,a]_{[-2,-1]})$ and $\omega_2\in \Pre([a,b]_{[-2,-1]})$. We have $\omega_1\cdot b\in \Pre([a,b]_{[-2,-1]})\subset \Pre(C_1)$ while $\omega_2\cdot b\in \Pre([b,b]_{[-2,-1]})\subset \Pre(C_2)$. 
Thus, for any $C\in \cC$ property~\eqref{eq:Propagation2} is not satisfied, implying that  $\cS_i(C_1)=\emptyset$, which is not possible since $\Post(C_1)=\Sigma^+$ and thus contradicting \eqref{eq:Propagation1}.
Intuitively, the set $\cC$, while providing a covering of the full shift $\Sigma^\Z$, is \emph{not} a $\g$-covering, since we cannot concatenate/propagate uniquely the signals of its classes. 
\end{example}

We now see how the conditions of Proposition~\ref{prop:PropertiesOFFP} characterize $\g$-coverings and we show that they can be inverted to provide a graph presentation of a given covering.
\begin{lemma}\label{defn:GraphAssociatedToCovering}
Consider a sofic shift $Z\subset \Sigma^\Z$, and a covering $\cC=\{C_1,\dots C_K\}$ of $Z$. Suppose that there exist  functions $\cS_i:\cC\to  2^{\cC}$ satisfying conditions~\eqref{eq:Propagation1}-\eqref{eq:Propagation2} of Proposition~\ref{prop:PropertiesOFFP}.
Then, defining the corresponding graph $\cG_{\cS}=(S_\cS,E_\cS)$ as follows:
\[
\begin{aligned}
S_\cS=\{s_C\}_{C\in \cC},\\
(s_C,s_D,i)\in E_\cS \;\;\text{if and only if }D\in \cS_i(C),
\end{aligned}
\]
we have that $\cC$ is a $\g$-covering and $\cG_\cS$ is a graph presentation of $\cC$.
\end{lemma}
\begin{proof}
Consider $C\in \cC$, we want to prove that, given any $\z\in Z$, we have $\z\in C$ if and only if there exists a bi-infinite walk $\pi$ in $\cG_\cS$ starting at $s_C$ and labeled by $\z$.

 We first prove that, for every $K\in \N$ and any $\wi \in \Sigma^K$, we have $[\wi]_{[0,K-1]}\cap C\neq \emptyset$ if and only if there exists a finite forward path in $\cG_\cS$ starting at $s_C$ and labeled by $\wi$. Let us consider the case $K=1$ first, consider $i\in \Sigma$, recalling~\eqref{eq:Propagation1} we have
\[
[i]_{[0,0]}\cap C\neq \emptyset\;\;\Leftrightarrow\;\;\exists \,D\in \cC\text{ such that } D\in \cS_i(C).
\]
The inductive step is obtained by similar argument.
Using~\eqref{eq:Propagation2} one can similarly prove that for every $K\in \N$ and any $\wi \in \Sigma^K$, it holds that $[\wi]_{[-K,-1]}\cap C\neq \emptyset$ if and only if there exists a finite backward path in $\cG_\cC$ starting at $s_C$ and labeled by $\wi$, concluding the proof.
\end{proof}
On the other hand, the following example shows that a graph presentation of a $\g$-covering is not unique, in  general.
\begin{example}\label{example:MultipleGraphRepr}
Consider $\Sigma=\{a,b\}$, $Z=\Sigma^\Z$, and $\cC=\{C_1,C_2\}$ with $C_1=C_2=\Sigma^\Z$. It can be seen that the three graphs in Figure~\ref{figure:Figure2} are valid (and non-isomorphic\footnote{The graphs  $\cG_1=(S,E)$ and $\cG_2=(Q,F)$ are said to be \emph{isomorphic} if there exists a bijection $\varphi:S\to Q$ such that $(s,q,i)\in E$ if and only if $(\varphi(s),\varphi(q),i)\in F$.}) graph presentations of $\cC$.
\end{example}

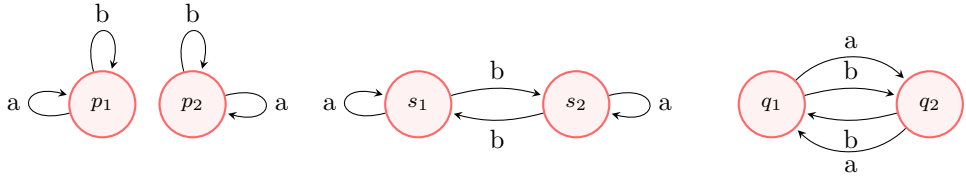
\begin{figure}[b!]
    \color{black}
    \centering
    \begin{tikzpicture}%
    [>=stealth,
    shorten >=1pt,
    node distance=0.5cm,
    on grid,
    auto,
    every state/.style={draw=red!60, fill=red!5, thick}
    ]
    \node[state] (left)   []               {\footnotesize $s_1$};
    \node[state] (right) [right=of left, xshift=1.6cm] {\footnotesize $s_2$};
      \node[state] (leftleft) [left=of left, xshift=-2.5cm] {\footnotesize $p_2$};
         \node[state] (leftleftleft) [left=of leftleft, xshift=-0.7cm] {\footnotesize $p_1$};
\node[state] (right1) [right=of right, xshift=2.1cm] {\footnotesize $q_1$};
         \node[state] (right2) [right=of right1, xshift=1.6cm] {\footnotesize $q_2$};

    \path[->]
    %   FROM       BEND/LOOP           POSITION OF LABEL   LABEL   TO
    %SecondGraph
    (left) edge[loop left=60]     node                      {a} (left)
    (right) edge[bend right=-15]     node                      {b} (left)
     (left) edge[bend right=-15]     node                      {b} (right)
    (right) edge[loop right=300]     node                      {a} (right)
%First graph
    (leftleftleft) edge[loop left=60]     node                      {a} (leftleftleft)
     (leftleftleft) edge[loop above=90]     node                      {b} (leftleftleft)
      (leftleft) edge[loop right=60]     node                      {a} (leftleft)
     (leftleft) edge[loop above=90]     node                      {b} (leftleft)
 %Third Grapj
 (right1) edge[bend right=-15]     node                      {b} (right2)
    (right1) edge[bend right=-45]     node                      {a} (right2)
     (right2) edge[bend right=-15]     node                      {b} (right1)
    (right2) edge[bend right=-45]     node                      {a} (right1)

;
    \end{tikzpicture}
\caption{Three possible graph presentations of the covering of $\{a,b\}^\Z$ in Example~\ref{example:MultipleGraphRepr}. }\label{figure:Figure2}
\end{figure}
Uniqueness indeed holds when considering non-redundant coverings, as proven in what follows.
\begin{lemma}\label{lemma:Unicity}
Consider a sofic shift $Z\subset \Sigma^\Z$, and a non-redundant $\g$-covering $\cC=\{C_1,\dots C_K\}$ of $Z$. Then there exists a unique (up to isomorphism) graph presentation of $\cC$, that we denote by $\cG_\cC$.
\end{lemma}
\begin{proof}
Suppose by contradiction that $\cH_1=(S,E)$ and $\cH_2=(Q,F)$ are two non-isomorphic graph presentations of~$\cC$. For every $C\in \cC$, denote by $s_C\in S$ and $q_C\in Q$ the corresponding nodes in $\cH_1$ and $\cH_2$, respectively. Since by assumption $\cH_1$ and $\cH_2$ are not isomorphic, consider $C,D\in \cC$ and $i\in \Sigma$ such that
\[
(s_C,s_D,i)\in E \;\;\wedge \;\;(q_C,q_D,i)\notin F.
\]
Since $(s_C,s_D,i)\in E$, we can consider $\z\in C$ such that $z_0=i$ and $\sigma(\z)\in D$. Since $\sigma(\z)\in D=\cZ(\cH_2,q_D)$, by~\eqref{eq:Propagation2} there must exist $B\neq C$ such that $(q_B,q_D,i)\in F$ and $\z\in \cZ(\cH_2,q_B)=B$, contradicting the non-redundancy.
\end{proof}

\subsection{Finite-Covering Lyapunov functions}\label{subsection:Finite-Covering}
In this subsection we show how $\g$-coverings (and thus, graphs)  can provide a tool to refine Definition~\ref{defn:SequenceDepLyapFunct}, having more tractable Lyapunov criteria for stability of~\eqref{eq:System}. 
In the switched case, the conditions turn out to be a \emph{finite} set of inequalities, providing algorithmically appealing sufficient conditions for stability. 
\begin{defn}\label{defn:FiniteGraphLyapunov}
Given a sofic shift $Z$, consider a $\g$-covering of $Z$, given by $\cC=\{C_1,\dots, C_K\}$. Consider $\cG=(S,E)$, a graph presentation of $\cC$ and a function $f:\R^n\times Z\to \R^n$. A $\cG$-based Lyapunov function is a $W:\R^n\times \cC\to \R$ such that there exist $\alpha_1,\alpha_2\in \cK_\infty$ and $\gamma\in [0,1)$ such that
\begin{subequations}
\begin{equation}\label{Eq:Sandwich2}
\alpha_1(|x|)\leq W(x,C)\leq \alpha_2(|x|),\;\;\forall x\in \R^n,\;\forall C\in \cC,
\end{equation}
\begin{equation}\label{eq:Decreasing2}
W(f(x,\z),D)\leq \gamma W(x,C),\;\;\;\begin{aligned}&\forall x\in \R^n,\;\forall C,D \in\cC,\forall i\in \Sigma\;\text{ s.t. } (s_C,s_D,i)\in E,\\ &\forall\, \z\text{ s.t. } \z\in C\;\wedge\;z_0=i\;\wedge\; \sigma(\z)\in D.
\end{aligned}
\end{equation}
\end{subequations}
\end{defn}
For any presentation $\cG$ of $Z$, we now prove that any $\cG$-based Lyapunov function implicitly defines a sd-LF (as introduced in Definition~\ref{defn:SequenceDepLyapFunct}), whose existence in turns completely characterizes the GUAS property (recall Theorem~\ref{lemma:ConverseNonLinear}).

\begin{thm}\label{prop:FromFiniteToInfinite}
Consider any $\g$-covering $\cC=\{C_1,\dots,C_K\}$ of a sofic shift $Z$, and $\cG$ a presentation of~$\cC$. Suppose $W:\R^n\times \cC\to \R$ is a $\cG$-based Lyapunov function for system~\eqref{eq:System}. Then the functions $ \cV_{\min}$ and $\cV_{\max}:\R^n\times Z\to \R$ defined by
\begin{subequations}
\begin{equation}\label{eq:LyapRedMin}
\cV_{\min}(x,\z):=\min_{\substack{X\in \cC\\ \z\in X}}W(x,X),
\end{equation}
\begin{equation}\label{eq:LyapRedMax}
\cV_{\max}(x,\z):=\max_{\substack{X\in \cC\\ \z\in X}}W(x,X),
\end{equation}
\end{subequations}
are sd-LF, in the sense of Definition~\ref{defn:SequenceDepLyapFunct}.
\end{thm}
\begin{proof}
By inequality~\eqref{Eq:Sandwich2} it follows that the  functions $\cV_{\min}$ and $\cV_{\max}$ satisfy inequality~\eqref{Eq:Sandwich1}.
 Now, consider any $\z\in Z$ and suppose $z_0=i\in \Sigma$. Consider any $C\in \cC$ such that $\z\in C$, thus implying that $C\cap [i]_{[0,0]}$ is not empty. By~\eqref{eq:SuffPropagation1}-\eqref{eq:SuffPropagation2} there exists $D\in \cC$ such that $D\in \cS_i(C)$ and $\sigma(\z)\in D$. Since $\z\in Z$, $i\in \Sigma$, and $C\in \cC$ were arbitrary, we have proven that
\begin{equation}\label{eq:technicalforward}
\forall C\in \cC,\;\forall i\in \Sigma,\; \forall \z\in C\cap [i]_{[0,0]}, \;\;\exists D\in \cS_i(C)\text{ such that } \sigma(\z)\in D.
\end{equation}
With a similar backward reasoning, one can see that 
\begin{equation}\label{eq:technicalbackward}
\forall D\in \cC,\;\forall i\in \Sigma,\; \forall \z\in D\cap [i]_{[-1,-1]}, \;\;\exists C\in \cS^{-1}_i(D)\text{ such that } \sigma^{-1}(\z)\in C.
\end{equation}
Given any $i\in \Sigma$, consider any $\z\in Z$, with $z_0=i$ and any $x\in \R^n$ and consider $C\in \cC$ such that $\min_{\substack{X\in \cC\\ \z\in X}}W(x,X)=W(x,C)$ and a $D\in \cC$ such that~\eqref{eq:technicalforward} holds. Using~\eqref{eq:Decreasing2}, we have
\[
\cV_{\min}(f(x,\z),\sigma(\z))=\hspace{-0.15cm}\min_{\substack{X\in \cC\\ \sigma(\z)\in X}}W(f(x,\z),X)\leq W(f(x,\z), D)\leq \gamma W(x,C)=\hspace{-0.05cm}\gamma \cV_{\min}(x,\z)
\]
proving~\eqref{eq:Decreasing1} for $\cV_{\min}$.
For the function $\cV_{\max}$ the reasoning is similar: consider any $\z\in Z$ and suppose $\cV_{\max}(f(x,\z),\sigma(\z))=W(f(x,\z),D)$ for some $D\in \cC$ such that $\sigma(\z)\in D$. Consider a $C\in \cC$ such that~\eqref{eq:technicalbackward} holds, computing we have
\[
\cV_{\max}(f(x,\z),\sigma(\z))=W(f(x,\z),D)\leq \gamma W(x, C)\leq \gamma \max_{\substack{X\in \cC\\ \z\in X}}W(x,X)=\cV_{\max}(x,\z),
\]
concluding the proof.
\end{proof}
We note that if the considered $\g$-covering is non-redundant, then the definitions in~\eqref{eq:LyapRedMin} and~\eqref{eq:LyapRedMax} coincide with the simple identification $\cV(x,\z)=W(x,X)$, for the unique $X\in \cC$ such that $\z\in X$.
\begin{rmk}[Switched systems case]
In the switched systems case, i.e. considering functions $f:\R^n \times \Sigma\to \R^n$, condition~\eqref{eq:Decreasing2} reads
\[
W(f(x,i),D)\leq \gamma W(x,C),\;\;\forall x\in \R^n,\forall C,D \in\cC,\forall i\in \Sigma\;\text{ s.t. } (s_C,s_D,i)\in E,
\]
and thus Definition~\ref{defn:FiniteGraphLyapunov} imposes, jointly with the positive definiteness conditions in~\eqref{Eq:Sandwich2}, a finite number of  inequalities among the functions $W(\cdot,C)$, $C\in \cC$, one for each edge $e\in E$. Thus, in this setting, Definition~\ref{defn:FiniteGraphLyapunov} provides a multiple-Lyapunov sufficient condition for stability, with an underlying graph whose structure defines the required inequalities. In the switched systems literature (see~\cite{AJPR:14,PEDJ:16,PhiAth19} and references therein), functions satisfying the conditions in Definition~\ref{defn:FiniteGraphLyapunov} are referred to as \emph{path-complete Lyapunov functions}. 
In the next statement, we prove that in this case, as far as the general class of continuous functions is considered, finite-covering Lyapunov functions provide a ``finite'' characterization of stability, to be compared with the ``infinite'' conditions given in Theorem~\ref{lemma:ConverseNonLinear}.
\end{rmk}

\begin{thm}[Finite-covering Converse Lyapunov Theorem for Switched Systems]\label{lemma:FiniteCoveringLemma}
Consider any sofic shift $Z\subseteq \Sigma^\Z$, any $f:\R^n\times \Sigma\to \R^n$, any $\g$-covering defined by $\cC=\{C_1,\dots,C_K\}$ of $Z$, and consider $\cG$ a presentation of $\cC$.
System~\eqref{eq:System} is GUAS (in the sense of Definition~\ref{defn:GUAS}) if and only if there exists a $\cG$-based Lyapunov function for system~\eqref{eq:System}.
\end{thm}
The proof substantially follows the ideas of  Theorem~\ref{lemma:ConverseNonLinear}, and for completeness is reported in Appendix~\ref{appendix:ProofFinite}.
Theorem~\ref{lemma:FiniteCoveringLemma} provides a ``finite'' counterpart of Theorem~\ref{lemma:ConverseNonLinear}: for switched GUAS systems, a finite covering Lyapunov function exists, no matter the chosen $\g$-covering of the sofic shift. On the other hand, when restricting the search to particular subclasses of continuous functions (e.g. $\cL_1$-weighted norms, quadratic functions, SOS polynomials), the size and topology of the considered $\g$-covering (equivalently, graph presentation) will play a crucial role in defining the level of conservatism of the arising Lyapunov conditions. For this reason, in the subsequent sections we further analyze the relations between graph-topology and stability conditions, introducing remarkable and numerically appealing hierarchies of graphs.

\subsection{Properties of Graphs vs Properties of Coverings, Duality}
In Subsection~\ref{subsection:GraphAndPartion} we studied the formal correspondence between graphs and the related coverings of sofic shifts they induce; then in Subsection~\ref{subsection:Finite-Covering} we provided the definition and main results concerning Lyapunov functions. As a by-product, we also provided a formal equivalence of different frameworks in the context of stability analysis of switched systems: results concerning graph-based Lyapunov functions~\cite{AJPR:14,PEDJ:16,PhiAth19,CHITOUR2021101021,DelPas22,AazanGir22} can be interpreted as conditions implicitly based on coverings of the underlying sofic shift. This formal correspondence was only sketched in the preliminary~\cite{DelRosJun23a,DelRosJun23b}.

In this subsection we continue the analysis of this correspondence, providing new results and insights. We present how graph and covering properties are related one to another, we underline the applications to the stability analysis of~\eqref{eq:System} and we show how this connection can be leveraged for algorithmic purposes in Lyapunov analysis.

\begin{defn}[Time-Inversion and Time-Inverse Covering]\label{defn:TimeInvCovering}
Given any set $S\subset \Sigma^\Z$ by $S^{-1}$ we denote its time-inversion, defined by
\[
S^{-1}:=\{\z^{-1}\;\vert\;\z\in S\}.
\]
Given a sofic shift $Z$ and a covering $\cC=\{C_1,\dots,C_K\}\subset\cP(Z)$, we define the \emph{time-inverse} of $\cC$, by $\cC^{-1}=\{C_1^{-1},\dots,C_K^{-1}\}$, which is a covering of $Z^{-1}$. 
\end{defn}
We state in what follows a lemma characterizing inverse coverings and their presentations.
\begin{lemma}[Time-Inversion and Transpose Graphs]\label{lemma:TimeInversion}
Given any sofic shift $Z$, consider a $\g$-covering $\cC$, its time inversion $Z^{-1}$ and the time-inverse covering $\cC^{-1}$. A graph $\cG$ is a graph presentation of $\cC$ if and only if $\cG^\top$ is a graph presentation of $\cC^{-1}$. Moreover, $\cC$ is a (non-redundant) $\g$-covering if and only if $\cC^{-1}$ is so.
\end{lemma}
\begin{proof}
Consider $\cG=(S,E)$ a graph presentation of $\cC$, with $S=\{s_C\}_{C\in \cC}$. Consider any $C\in \cC$ and any $\z\in C$, by definition there exist a bi-infinite walk $\pi=(\dots, e_{-1},e_0,e_1,\dots)$ in $\cG$ starting at $s_C$ and labeled by $\z$. It is easy to see that $\pi^{-1}=(\dots, e_0,e_{-1},e_{-2}, \dots)$ is a bi-infinite walk in $\cG^\top$ starting at $s_C$ and it is labeled by $\z^{-1}\in C^{-1}$. By arbitrariness of $C\in \cC$ and  $\z\in C$ and recalling that, for any set $C\subseteq\Sigma^\Z$  and for any graph we have $(C^{-1})^{-1}=C$ and $(\cG^\top)^\top=\cG$, we conclude.
\end{proof}

In what follows we identify and study important subclasses of coverings, which were already considered in the literature.

\begin{defn}[Particular case: Memory and Future coverings]
A $\g$-covering $\cC\subset \cP(Z)$ of a sofic shift is said to be a \emph{memory covering} if $C=\Pre(C)\cdot {\Post}(Z)$  for all $C\in \cC$.
Similarly, a $\g$-covering $\cC\subset \cP(Z)$ is said to be a \emph{future covering} if $C=\Pre(Z)\cdot\Post(C)$  for all $C\in \cC$. We note that $\cC$ is a memory covering if and only  if $\cC^{-1}$ is a future covering.
\end{defn}

In the following we formally characterize memory and future $\g$-coverings of the full shift $\Sigma^\Z$.

\begin{prop}\label{Prop:GraphsVsMemort}
Consider $\cC$  a $\g$-covering of $\Sigma^\Z$.  We have that 
\begin{enumerate}[label=(\Alph*):]
\item $\cC$ is a memory covering if and only if any $\cG=(S,E)$ graph  presentation of $\cC$ is complete.
\item If $\cC$  is a memory and non-redundant covering, then $\cG_\cC=(S,E)$, the graph presentation of $\cC$,  is complete and deterministic.
\item $\cC$ is a future covering if and only if any $\cG=(S,E)$ graph presentation of $\cC$  is co-complete.
\item If $\cC$ is a future and non-redundant covering then $\cG_\cC=(S,E)$, the graph presentation of $\cC$, is co-complete and co-deterministic.
\end{enumerate}
\end{prop}
\begin{proof}
\emph{(A):} Let us consider any graph presentation $\cG$ and suppose it is complete. By completeness, for any $s\in S$ and any $i\in \Sigma$ there exist $q\in S$ such that $(s,q,i)\in E$. That means that $\Post(C_s)\cap [i]_{[0,0]}\neq \emptyset$. Iterating the reasoning forward, since each node admits outgoing edges labeled by any $i\in \Sigma$,  it can be seen that $\Post(C_s)=\Sigma^+$, and thus $C_s=\Pre(C_s)\cdot \Sigma^+$; by arbitrariness of $s\in S$ we conclude.

Now suppose $\cC$ is a memory covering, that is, for any $C\in \cC$ we have $\Post(C)=\Sigma^+$. Thus, for any $i\in \Sigma$, $\Post(C)\cap [i]_{[0,0]}\neq \emptyset$ and thus by~\eqref{eq:Propagation1} there exists $D\in \cC$ such that $(s_C,s_D,i)\in E$, for any $\cG=(S,E)$ graph presentation of $\cC$.

\emph{(B):} Suppose that $\cC$ is a memory covering and non-redundant and consider $\cG_\cC$ the graph presentation of $\cC$ (recall Lemma~\ref{lemma:Unicity}). From Item \emph{(A)} we know that $\cG_\cC$ is complete, we now prove that it is also deterministic.  Suppose by contradiction that there are $q_1\neq q_2\in S$ such that $(s,q_j,i)\in E$ for $j\in \{1,2\}$. Then given any word $\z\in C_s=\Pre(C_s)\cdot \Sigma^+$ with $z_0=i$ we would have, by definition, $\sigma(\z)\in C_{q_1}\cap C_{q_2}$, contradicting the non-redundancy. 

\noindent
\emph{(C),(D):} Trivially follow by Items \emph{(A)} and \emph{(B)}, using Lemma~\ref{lemma:TimeInversion}.
\end{proof}
\paragraph{Dual Functions for Linear Systems}
We now show that, in the linear case, the correspondence between time-inversion and transpose graphs can provide a tool for stability analysis of dual dynamical systems, defined in what follows.
\begin{defn}[Dual System]
Consider a sofic shift $Z$ and a linear system, i.e. the function $f:\R^n\times Z\to \R^n$ is defined by $f(x,\z)=A(\z)x$ with $A:Z\to \R^{n\times n}$; let us denote the system by $(Z,A(\cdot))$. The \emph{dual system}, denoted by $[Z,A(\cdot)]^\top$, is defined by $(Z^{-1},A^*(\cdot))$, where $A^*(\z^{-1})=A(\z)^\top$, for all $\z\in Z$.
\end{defn}
We now consider, for any $n\in \N$, the set $\cN_n:=\{v:\R^n\to \R\;\vert\;v\text{ is a norm}\}$, the set of all the norms on $\R^n$.
We need a definition of dual norm, recalled in what follows.
\begin{defn}[Dual Norm]
Consider any norm $v\in \cN_n$, the dual norm of $v$, denoted by $v^\star$, is defined by 
\[
v^\star(x):=\sup_{y\in \R^n,\;v(y)=1}y^\top x\;\;\;\;\forall x\in \R^n.
\]
It can be seen that $v^\star\in \cN_n$.
\end{defn}
For the formal definition and further discussion on duality theory and dual norms we refer to~\cite[Part III]{RockConv}. We have the following duality result.
\begin{prop}[Dual Lyapunov functions]\label{prop:Duality}
Consider a sofic shift $Z$ and a linear systems defined by $f(x,\z)=A(\z)x$ with $A:Z\to \R^{n\times n}$. Consider a $\g$-covering $\cC=\{C_1,\dots,C_N\}$, a graph presentation $\cG=(S,E)$ of $\cC$, and  a $\cG$-based Lyapunov function  $W:\R^n\times \cC\to \R$ such that
$W(\cdot\,,C)\in \cN_n$, for all $C\in \cC$. Then the function $W^\top:\R^n\times \cC^{-1}\to \R$ defined by $W^\top(\cdot\,, C^{-1})=W^\star(\cdot\,, C)$ is a $\cG^\top$-based Lyapunov functions for the dual system defined by $(Z^{-1}, A^\star(\cdot))$.
\end{prop}
\begin{proof}
The fact that $W^\top$ satisfies an inequality of the form~\eqref{Eq:Sandwich2} is trivial.
Then, the main tool for the proof is the following well-known duality result: Given any $v,w\in \cN_n$, any $A\in \R^{n\times n}$ and any $\gamma\in \R_+$ we have
\[
v(Ax)\leq \gamma w(x),\;\forall\;x\in \R^n\;\Leftrightarrow\;\;w^\star(A^\top x)\leq \gamma v^\star(x),\;\;\forall\;x\in \R^n.
\]
Now, consider $C,D \in\cC$ and $i\in \Sigma\;\text{ s.t. } (s_C,s_D,i)\in E$ and $\z\in Z$ such that  $\z\in C\;\wedge\;z_0=i\;\wedge\; \sigma(\z)\in D$ (and thus for which inequality~\eqref{eq:Decreasing2} holds). This is equivalent to
$(s_{D^{-1}},s_{C^{-1}},i)\in E^\top$, $\sigma(z)^{-1}\in D^{-1}$, $\z^{-1}\in C^{-1}$, $(\sigma(z))^{-1}_0=i$. Moreover note that $\sigma(\sigma(z)^{-1})=z^{-1}$. We have to verify
\[
W^\top(A^\star(\z^{-1})x,C^{-1})=W^\top(A(\z)^\top x,C^{-1})\leq \gamma W^\top(x,D^{-1}),\;\;\forall x\in \R^n,
\]
which is equivalent, by duality, to 
\[
W(A(\z)x, D)\leq \gamma W(x,C),\;\;\forall x\in \R^n,
\]
which holds by assumption, thus concluding the proof.
\end{proof}
We note that in the duality result in Proposition~\ref{prop:Duality}, other classes of convex functions (with respect to the class of norms) can be considered. For example, an equivalent result can be stated for positive definite, strictly convex and homogeneous of degree~$2$ functions, on the lines of what is done in~\cite{GoeHu06}.  We decided, for readability concerns, to present only the result involving norms.

 In this subsection, summarizing, we clarified the formal correspondence between time-inversion of the coverings, graph transposition and duality theory of linear systems. This allowed us to provide a novel algebraic construction of Lyapunov functions for one system, if one knows a solution for the time-inverted/dual system. This generalizes the observations provided, for the switched systems case, in~\cite{EssickLee14,AJPR:14}. For linear systems, we showed how this connects with the theory of duality of convex analysis. In the subsequent Section~\ref{sec:LinearAndNumerical} we will see how this can be leveraged for algorithmic purpose.

\subsection{Shifts of Finite Type: Generalized De Bruijn graphs and Asymptotic Converse Lyapunov Lemma}
In this subsection we provide a new class of graphs, presentations of a given sofic shift, which will provide canonical candidate structures of the Lyapunov construction in~Definition~\ref{defn:SequenceDepLyapFunct}. In order to allow this construction, we need a refinement of the definition of sofic shift.

\begin{defn}[Shift of $M$-Finite-Type]\label{Def:ShiftFiniteTYpe}
Given $M\in \N$, a shift $Z\subset \Sigma^\Z$ is said to be of \emph{$M$-finite type} if and only if it satisfies the following property:
\[
\forall \wi,\wh,\wj\in \Sigma^\star,\;\;\left(
\wi \cdot \wh\in \cL(Z)\;\wedge \wh\cdot\wj\in \cL(Z) \;\wedge\;|\wh|\geq M\;\Rightarrow \wi\cdot\wh\cdot\wj \in \cL(Z)\right).
\]
\end{defn}
Intuitively, a shift is of  $M$-finite type if it suffices to check subwords of length $M+1$ to decide if a word lies in $\cL(Z)$, or, equivalently, if the set of forbidden words can be constructed by forbidden words of length at most $M+1$. Moreover, a shift of finite type is in particular sofic, we refer to~\cite[Chapter 2]{LindMarcus95} for the formal discussion. Shifts of $0$-finite type are simply shifts of the form $\wt \Sigma^\Z$ for some $\wt \Sigma\subseteq \Sigma$. 
In the following we introduce a particular shift of $1$-finite type which we will also use in subsequent sections, as working example.
\begin{example}{\emph{(Golden Mean Shift)}~\cite[Example 1.2.3]{LindMarcus95}:}\label{ex:GoldenMean}
Given $\Sigma=\{a,b\}$, let us consider $Z\subset\Sigma^\Z$ the set of all the bi-infinite sequences with no two consecutive $a$'s. It can be seen that it is a shift of $1$-finite type.
\end{example}

 We use the notation: given any $K\in \N$ and any $\wi=(i_0,\dots i_{K-1})\in \Sigma^K$, we define $\wi^-:=(i_0,\dots, i_{K-2})$ and $\wi^+:=(i_1,\dots, i_{K-1})\in \Sigma^{K-1}$.
In the following we generalize the classical definition in~\cite{DeBru46}, in order to construct a covering that mixes information from memory and future.
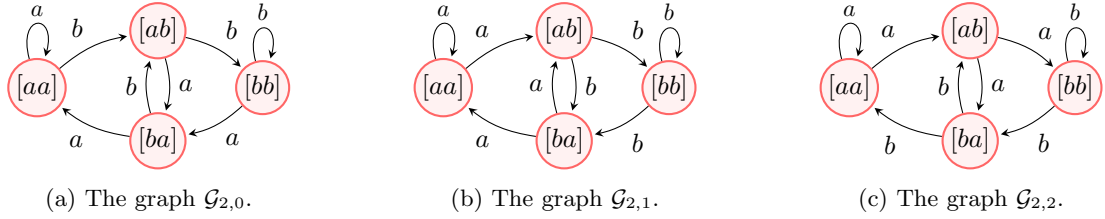
\begin{figure}[t!]
\vspace{1cm}
\begin{subfigure}{0.32\linewidth}
  \color{black}
  \centering
  %\resizebox{5.7cm}{4.2cm}{%
  \begin{tikzpicture}%
  [>=stealth,
  shorten >=1pt,
  node distance=1cm,
  on grid,
  auto,
  every state/.style={draw=red!60, fill=red!5, thick}
  ]
  \node[state, inner sep=1pt, minimum size=20pt] (left)                  {$[aa]$};
  \node[state, inner sep=1pt, minimum size=20pt] (right) [right=of left, xshift=2cm] {$[bb]$};
  \node[state, inner sep=1pt, minimum size=20pt] (upper) [above right=of left, xshift=0.9cm, yshift=0.05cm]{$[ab]$};
  \node[state, inner sep=1pt, minimum size=20pt] (below) [below right=of left, xshift=0.9cm, yshift=0cm]{$[ba]$};
  \path[->]
  %   FROM       BEND/LOOP           POSITION OF LABEL   LABEL   TO
  (left) edge[loop above=60]     node             [scale=0.9]             {$a$} (left)
  (left) edge[bend left=15]     node                      {$b$} (upper)
  (upper) edge[bend left=15]     node                      {$a$} (below)
  (below) edge[bend left=15]     node                      {$a$} (left)
  (below) edge[bend left=15]     node                      {$b$} (upper)
  (right) edge[bend left=15]     node                      {$a$} (below)
  (upper) edge[bend left=15]     node                      {$b$} (right)
  (right) edge[loop above=300]     node            [scale=0.9]              {$b$} (right)
;
  \end{tikzpicture}
  %}
  \caption{The graph  $\cG_{2,0}$.}
  \end{subfigure}
  \begin{subfigure}{0.32\linewidth}
  \color{black}
  \centering
 \begin{tikzpicture}%
  [>=stealth,
  shorten >=1pt,
  node distance=1cm,
  on grid,
  auto,
  every state/.style={draw=red!60, fill=red!5, thick}
  ]
  \node[state, inner sep=1pt, minimum size=20pt] (left)                  {$[aa]$};
  \node[state, inner sep=1pt, minimum size=20pt] (right) [right=of left, xshift=2cm] {$[bb]$};
  \node[state, inner sep=1pt, minimum size=20pt] (upper) [above right=of left, xshift=0.9cm, yshift=0.05cm]{$[ab]$};
  \node[state, inner sep=1pt, minimum size=20pt] (below) [below right=of left, xshift=0.9cm, yshift=0cm]{$[ba]$};
  \path[->]
  %   FROM       BEND/LOOP           POSITION OF LABEL   LABEL   TO
  (left) edge[loop above=60]     node             [scale=0.9]             {$a$} (left)
  (left) edge[bend left=15]     node                      {$a$} (upper)
  (upper) edge[bend left=15]     node                      {$b$} (below)
  (below) edge[bend left=15]     node                      {$a$} (left)
  (below) edge[bend left=15]     node                      {$a$} (upper)
  (right) edge[bend left=15]     node                      {$b$} (below)
  (upper) edge[bend left=15]     node                      {$b$} (right)
  (right) edge[loop above=300]     node            [scale=0.9]              {$b$} (right)
;
  \end{tikzpicture}
   \caption{The graph  $\cG_{2,1}$.}
  \end{subfigure}
\begin{subfigure}{0.32\linewidth}
    \color{black}
  \centering
  %\resizebox{5.7cm}{4.2cm}{%
   \begin{tikzpicture}%
  [>=stealth,
  shorten >=1pt,
  node distance=1cm,
  on grid,
  auto,
  every state/.style={draw=red!60, fill=red!5, thick}
  ]
  \node[state, inner sep=1pt, minimum size=20pt] (left)                  {$[aa]$};
  \node[state, inner sep=1pt, minimum size=20pt] (right) [right=of left, xshift=2cm] {$[bb]$};
  \node[state, inner sep=1pt, minimum size=20pt] (upper) [above right=of left, xshift=0.9cm, yshift=0.05cm]{$[ab]$};
  \node[state, inner sep=1pt, minimum size=20pt] (below) [below right=of left, xshift=0.9cm, yshift=0cm]{$[ba]$};
  \path[->]
  %   FROM       BEND/LOOP           POSITION OF LABEL   LABEL   TO
  (left) edge[loop above=60]     node             [scale=0.9]             {$a$} (left)
  (left) edge[bend left=15]     node                      {$a$} (upper)
  (upper) edge[bend left=15]     node                      {$a$} (below)
  (below) edge[bend left=15]     node                      {$b$} (left)
  (below) edge[bend left=15]     node                      {$b$} (upper)
  (right) edge[bend left=15]     node                      {$b$} (below)
  (upper) edge[bend left=15]     node                      {$a$} (right)
  (right) edge[loop above=300]     node            [scale=0.9]              {$b$} (right)
;
  \end{tikzpicture}
  %}
  \caption{The graph  $\cG_{2,2}$.}
  \end{subfigure}
  \caption{ De Bruijn graphs of the full shift $\{a,b\}^\Z$ of order $2$.} \label{Fig:GeneralDBgraph}
\end{figure}

\begin{defn}[Generalized De-Bruijn graphs] \label{defn:GenDeBrunjii}
Consider the full shift $\Sigma^\Z$, and a $K\in \N$, and a $k\in [0,K]$.
The \emph{De-Bruijn graph of order $K$ and position $k$}, denoted by $\cG_{K,k}=(S_{K,k},E_{K,k})$ defined by
\begin{equation}\label{eq:DefinitionDeBrunji}
\begin{aligned}
S_{K,k}&:=\{\wi \in \Sigma^K\},\\
(\wi,\wj, h)\in E_{K,k}\;\;&\text{iff }\begin{cases}
j_{K-1}=h\;\wedge\;\wi^+=\wj^-,  & \text{ if }k=0,\\
i_{K-k}=h\;\;\;\wedge\;\wi^+=\wj^-,  &\text{ if }k \in[1,K].
\end{cases} 
\end{aligned}
\end{equation}
Consider $M\in \N$ and any shift of $M$-finite type $Z\subset \Sigma^\Z$. Consider any $K\geq M$, and any $k\in [0,K]$, then the \emph{$Z$-De-Bruijn graph of order $K$ and position $k$}, denoted by $\cG_{K,k}(Z)=(S_{K,k}(Z),E_{K,k}(Z))$ is defined as in~\eqref{eq:DefinitionDeBrunji}, with $S_{K,k}(Z):=\{\wi \in \Sigma^K\cap \cL(Z)\}$.
\end{defn}
The graphs of the form $\cG_{K,0}$ for $K\in \N$ have been introduced in the seminal paper~\cite{DeBru46}. We have generalized here this definition (introducing the \emph{position} $k\in [0,K]$) in order to handle more general coverings of shifts of finite type, as we analyze in what follows. 
Figure~\ref{Fig:GeneralDBgraph} represents the three De Bruijn graphs of order $2$ of the full shift $\{a,b\}^\Z$, while the De Bruijn graphs of order $1$ and $2$ for the golden mean shift defined in Example~\ref{ex:GoldenMean} are depicted in Figure~\ref{Fig:Ex1DBgraph}. Note that the graph already encountered in Figure~\ref{fig:FirstDeBrujin} is $\cG_{1,0}$.
In the following we present some important properties of the De Bruijn graphs.
\begin{prop}[Properties of De Bruijn Graphs]\label{prop:propDeBrunji}
{}Given any  $M\in \N$ and any shift of $M$-finite type $Z\subset \Sigma^\Z$ any $K\geq M$, and any $k\in [0,K]$. Then:
\begin{enumerate}[leftmargin=*]
\item  $\cZ(\cG_{K,k}(Z))=Z$.\label{item:DeBRunji1}
\item  $\cG_{K,k}(Z)=(\cG_{K,K-k}(Z^{-1}))^\top$; and thus in particular,  $\cG_{K,k}=\cG_{K,K-k}^\top$.\label{item:DeBRunji2}
\item $\cG_{K,k}(Z)$ is the graph presentation of the (not-redundant) $\g$-covering of $Z$ defined by 
\[
\cC:=\{[\wi]_{[-K+k,k-1]}\;\vert\;\wi\in \Sigma^K\cap\cL(Z)\}.
\]\label{item:DeBRunji2.5}
\item In particular, for all $K\in \N$, $\cG_{K,0}$ induces a memory covering and  $\cG_{K,K}$ induces a future covering. \label{item:DeBRunji3}
%\item The canonical deterministic reduction of a De Bruijn graph is a (memory) De Bruijn graph, more precisely $\cO^\star_{\cG_{K,k}}=\cG_{K-k,0}$. Similarly, $\ccO^\star_{\cG_{K,k}}=\cG_{K-k,K-k}$. (\textcolor{red}{I don't know if this is important/interesting.}) \label{item:DeBRunji4}
\end{enumerate}
\end{prop}
\begin{proof}
Consider any $K\geq M$, and any $k\in [0,K]$.
To prove Item~\ref{item:DeBRunji1}, we note that, by construction $\cL(\cZ(\cG_{K,k}(Z)))=\cL(Z)$, since in~\eqref{eq:DefinitionDeBrunji} we concatenate words of length of at least $M$, recall the Definition~\ref{Def:ShiftFiniteTYpe}. Thus, by Lemma~\ref{Lemma:EquivalenceLanguage}, we conclude.
To prove Item~\ref{item:DeBRunji2}, given $\wi=(i_0,\dots, i_{K-1})\in \Sigma^K$, let us denote by $\wi^\top=(i_{K-1},\dots, i_0)$. It is easy to see that
\[
(\wi,\wj,h)\in E_{K,k}(Z) \;\;\Leftrightarrow\;\;(\wj^\top,\wi^\top,h)\in E_{K,K-k}(Z^{-1}),
\]
Indeed,  we note that for any $\wi\in \Sigma^K\cap\cL(Z)$, $\wi^\top_{k-1}=\wi_{K-k}$, for any $k\in [1,K]$. Considering $k\in [1,K]$, we have that $(\wi,\wj,h)\in E_{K,k}(Z)$ if and only if $\wi^+=\wj^-\wedge \wj_{k-1}=h$ which is equivalent to $(\wj^\top)^+=(\wi^\top)^-\;\wedge \;\wj_{K-k}=h$ which in turn is equivalent to $(\wj^\top,\wi^\top,h)\in E_{K,K-k}(Z^{-1})$. The case $k=0$ is similar.

For Item~\ref{item:DeBRunji2.5}, consider any $\wi=(i_0,\dots, i_{K-1}) \in \Sigma^K\cap\cL(Z)$, and consider any $\z\in \Sigma^\Z\cap [\wi]_{-K+k,k-1}$, we have to show that there exist a bi-infinite walk in $\cG_{K,k}$ labeled by $\z$ and starting at $\wi$.  Let us proceed by steps: suppose $\z(k)=h$, and, since $\z\in[\wi]_{-K+k,k-1}$ we have $\z(0)=i=i_{K-k}$. Consider $\wj=\wi^+\cdot h=(i_1,\dots, i_{K-1},h)$. By definition, we have $(\wi,\wj,h)\in E_{K,k}(Z)$. For the forward part, we can thus proceed iteratively, considering $\wj$ and $\sigma(\z)\in [\wj]_{[k-K,k-1]}$. The backward reasoning and the case $k=0$ follow similar argument and are avoided here. We have thus proven that $\cZ(\cG_{K,k},\,\wi)=[\wi]_{[-K+k,k-1]}$ and by arbitrariness of $\wi\in \Sigma^K\cap \cL(Z)$ we conclude.

For Item~\ref{item:DeBRunji3} we see that $\cG_{K,0}$ is complete and deterministic, and thus recalling Lemma~\ref{lemma:TimeInversion} it is the presentation of a memory covering, we conclude $\cG_{K,K}$ induces a future covering, by Item~\ref{item:DeBRunji2}.
\end{proof}

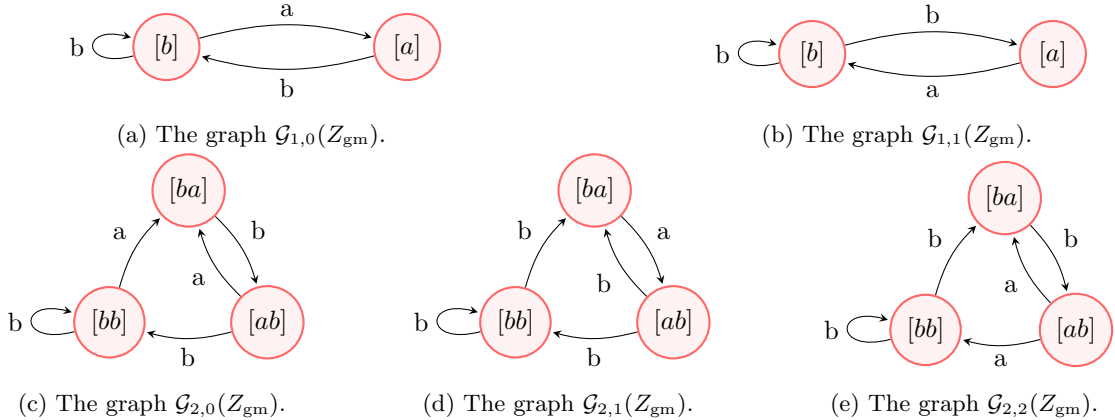
\begin{figure}[b!]
 \begin{subfigure}{0.48\linewidth}
    \color{black}
  \centering
  %\resizebox{5.7cm}{4.2cm}{%
  \begin{tikzpicture}%
      [>=stealth,
    shorten >=1pt,
    node distance=0.5cm,
    on grid,
    auto,
    every state/.style={draw=red!60, fill=red!5, thick}
    ]
  \node[state] (left)                  {$[b]$};
  \node[state] (right) [right=of left, xshift=2.7cm] {$[a]$};
  \path[->]
  %   FROM       BEND/LOOP           POSITbON OF LABEL   LABEL   TO
  (right)   edge[bend left=15] node                {b} (left)
  (left) edge[loop left=60]     node                      {b} (left)
(left)   edge[bend left=15] node                {a} (right)
  ;
  \end{tikzpicture}
  %}
  %}
  \caption{The graph $\cG_{1,0}(Z_{\text{gm}})$.}
  \end{subfigure}
\begin{subfigure}{0.48\linewidth}
    \color{black}
  \centering
  %\resizebox{5.7cm}{4.2cm}{%
  \begin{tikzpicture}%
      [>=stealth,
    shorten >=1pt,
    node distance=0.5cm,
    on grid,
    auto,
    every state/.style={draw=red!60, fill=red!5, thick}
    ]
  \node[state] (left)                  {$[b]$};
  \node[state] (right) [right=of left, xshift=2.7cm] {$[a]$};
  \path[->]
  %   FROM       BEND/LOOP           POSITION OF LABEL   LABEL   TO
  (right)   edge[bend left=15] node                {a} (left)
  (left) edge[loop left=60]     node                      {b} (left)
(left)   edge[bend left=15] node                {b} (right)
  ;
  \end{tikzpicture}
  %}
  \caption{The graph $\cG_{1,1}(Z_{\text{gm}})$.}
  \end{subfigure}
\begin{subfigure}{0.32\linewidth}
  \color{black}
  \centering
  %\resizebox{5.7cm}{4.2cm}{%
  \begin{tikzpicture}%
      [>=stealth,
    shorten >=1pt,
    node distance=0.5cm,
    on grid,
    auto,
    every state/.style={draw=red!60, fill=red!5, thick}
    ]
  \node[state] (left)                  {$[bb]$};
  \node[state] (right) [right=of left, xshift=1.6cm] {$[ab]$};
  \node[state] (upper) [above right=of left, xshift=0.7cm, yshift=1.4cm]{$[ba]$};
  \path[->]
  %   FROM       BEND/LOOP           POSITION OF LABEL   LABEL   TO
  (left) edge[bend left=15]     node                      {a} (upper)
  (right)   edge[bend left=15] node                {b} (left)
  (right)   edge[bend left=15] node                {a} (upper)
  (left) edge[loop left=60]     node                      {b} (left)
(upper)   edge[bend left=15] node                {b} (right)
  ;
  \end{tikzpicture}
  %}
  \caption{The graph  $\cG_{2,0}(Z_{\text{gm}})$.}
  \end{subfigure}
  \begin{subfigure}{0.32\linewidth}
    \color{black}
  \centering
  %\resizebox{5.7cm}{4.2cm}{%
  \begin{tikzpicture}%
      [>=stealth,
    shorten >=1pt,
    node distance=0.5cm,
    on grid,
    auto,
    every state/.style={draw=red!60, fill=red!5, thick}
    ]
  \node[state] (left)                  {$[bb]$};
  \node[state] (right) [right=of left, xshift=1.6cm] {$[ab]$};
  \node[state] (upper) [above right=of left, xshift=0.7cm, yshift=1.4cm]{$[ba]$};
  \path[->]
  %   FROM       BEND/LOOP           POSITION OF LABEL   LABEL   TO
  (left) edge[bend left=15]     node                      {b} (upper)
  (right)   edge[bend left=15] node                {b} (left)
  (right)   edge[bend left=15] node                {b} (upper)
  (left) edge[loop left=60]     node                      {b} (left)
(upper)   edge[bend left=15] node                {a} (right)
  ;
  \end{tikzpicture}
  %}
  \caption{The graph  $\cG_{2,1}(Z_{\text{gm}})$.}
  \end{subfigure}
\begin{subfigure}{0.32\linewidth}
    \color{black}
  \centering
  %\resizebox{5.7cm}{4.2cm}{%
  \begin{tikzpicture}%
      [>=stealth,
    shorten >=1pt,
    node distance=0.5cm,
    on grid,
    auto,
    every state/.style={draw=red!60, fill=red!5, thick}
    ]
  \node[state] (left)                  {$[bb]$};
  \node[state] (right) [right=of left, xshift=1.5cm] {$[ab]$};
  \node[state] (upper) [above right=of left, xshift=0.7cm, yshift=1.4cm]{$[ba]$};
  \path[->]
  %   FROM       BEND/LOOP           POSITION OF LABEL   LABEL   TO
  (left) edge[bend left=15]     node                      {b} (upper)
  (right)   edge[bend left=15] node                {a} (left)
  (right)   edge[bend left=15] node                {a} (upper)
  (left) edge[loop left=60]     node                      {b} (left)
(upper)   edge[bend left=15] node                {b} (right)
  ;
  \end{tikzpicture}
  %}
  \caption{The graph  $\cG_{2,2}(Z_{\text{gm}})$.}
  \end{subfigure}
  \caption{First De Bruijn graphs of the golden-mean shift $Z_{\text{gm}}$ in Example~\ref{ex:GoldenMean}.} \label{Fig:Ex1DBgraph}
\end{figure}
Consider again the three De Bruijn graphs of order $2$ of the shift $\{a,b\}^\Z$ depicted in Figure~\ref{Fig:GeneralDBgraph}.
Note that by Proposition~\ref{prop:propDeBrunji} the graph $\cG_{2,0}$  is the presentation of a \emph{memory} non-redundant covering, $\cG_{2,2}$ is the presentation of a \emph{future} non-redundant covering, while $\cG_{2,1}$ is the presentation of a covering that is neither a memory- nor future- covering. Note that the three graphs have the same number of nodes and edges and thus can be considered to be of the same complexity, in a sense that we clarify in following sections, in which we compare the arising Lyapunov conditions.

Inspired by the results in~\cite{AJPR:14,LeeDull06,Thesis:Essick}, we now show that De Bruijn graphs are a natural candidate for Lyapunov certificates for switched linear systems, when considering the class of quadratic functions.
\begin{thm}[Asymptotic Converse Lyapunov Theorem]\label{prop:DeBrunjiConverse}
Consider a shift of finite type $Z$ and a switched linear systems, i.e. the function $f:\R^n\times Z\to \R^n$ is defined by $f(x,\z)=A(z_0)x$ with $A:\Sigma\to \R^{n\times n}$. Let us fix  $\gamma\in [0,1)$. The following statements are equivalent:
\begin{enumerate}[leftmargin=0.6cm]
\item[$(1)$]  For any $\wt \gamma\in (\gamma,1)$, system~\eqref{eq:System} is UES  with decay rate $\wt \gamma$;
\item[$(2)$]  For any $\wt \gamma\in (\gamma,1)$ there exist  $K\geq 0$ and $k\in [0,K]$ such that
there exists a quadratic $\cG_{K,k}(Z)$-based Lyapunov function. More explicitly, there exist $M_1,M_2>0$ and $Q:\cL(Z)\cap \Sigma^K\to \mathbb{S}^n_{+}$ such that
\begin{subequations}
\begin{equation}\label{eq:FINITESandwichLinear}
 M_1I_n\preceq Q(\wi)\preceq M_2 I_n,\;\;\;\forall\;\wi\in S_{K,k}(Z),
\end{equation}
\begin{equation}\label{eq:FINITEPseudoRiccatti}
A(h)^\top Q(\wj)A(h)\prec \wt \gamma^2\, Q(\wi)\;\;\;\forall\;(\wi,\wj,h)\in E_{K,k}(Z).
\end{equation}
\end{subequations}
\end{enumerate}
\end{thm}
\begin{proof}
The proof substantially follows the ideas of  Theorem~\ref{Lemma:ConverseLinear}. 
We thus prove only $(1) \;\Rightarrow\;(2)$, and only for the future-based conditions, i.e. considering De Bruijn graphs of the form $\cG_{K,K}(Z)$. Suppose $Z$ is of $M$-finite type and consider any $K\geq M$.
Consider any $\wt \gamma\in (\gamma,1)$ and any $\wt \gamma_1\in (\gamma,\wt \gamma)$,  we define
\[
Q_+(\wi)=\sum_{k=0}^{K-1}\frac{1}{\wt \gamma^{2k}}S(k,\wi)^\top S(k,\wi),
\]
for any $\wi \in \cL(Z)\cap \Sigma^K$,  where $S(k,\wi)$ is the state-transition matrix corresponing to $\wi=(i_0,\dots, i_{K-1})\in  \cL(Z)\cap \Sigma^K$, i.e. $S(k,\wi):=A(i_{k-1})\cdot\cdot A(i_0)$, for any $k\leq K$, with the convention that $S(0,\wi)=I_n$ for all $\wi \in \cL(Z)\cap \Sigma^K$.
The positive definiteness in~\eqref{eq:FINITESandwichLinear} can be proven exactly as in Theorem~\ref{Lemma:ConverseLinear}. Let thus consider $(\wi,\wj,h)\in E_{K,K}(Z)$ i.e. $i_0=h$, and $\wi^+=\wj^-$, computing
\[
\begin{aligned}
A(h)^\top Q_+(\wj)A(h)&=\sum_{k=0}^{K-1}\frac{1}{\wt \gamma^{2k}}A(h)^\top S(k,\wj)^\top S(k,\wj)A(h)\\&=\wt \gamma ^2\sum_{k=0}^{K}\frac{1}{\wt \gamma^{2k}} S(k,h\cdot \wj)^\top S(k,h\cdot \wj)-\wt\gamma^2I_n\\&=\wt \gamma^2 Q_+(\wi)+\wt \gamma^2\frac{1}{\wt \gamma^{2K}} S(K,h\cdot \wj)^\top S(K,h\cdot \wj)-\wt\gamma^2I_n\\&\preceq\wt \gamma^2 Q_+(\wi)+\wt \gamma^2\left (\frac{M \wt \gamma_1^{2K}}{\wt \gamma^{2K}}-1  \right )I_n.
\end{aligned}
\]
Since $\wt \gamma_1<\wt\gamma$, we have that there exists a $K$ large enough such that $M\frac{\wt \gamma_1}{\wt \gamma}^{2K}\leq 1$, concluding the proof.
\end{proof}

\begin{rmk}
The Lyapunov conditions in Theorem~\ref{prop:DeBrunjiConverse}, in the restricted case of $\cG_{K,0}(Z)$ (resp. $\cG_{K,K}(Z)$) have been already introduced in~\cite{LeeDull06,DonDull20, EssickLee14,Thesis:Essick}. Indeed, these are natural conditions since they represent, as stated in Proposition~\ref{prop:propDeBrunji},  multiple Lyapunov functions which depend on a fixed number of previous (resp. future) values of the discrete state, i.e. they are presentations of the memory- (resp. future-) uniform coverings. Differently from the literature, Definition~\ref{defn:GenDeBrunjii} we introduce (the conditions arising from) generalized graphs of the form $\cG_{K,k}(Z)$ for $k\in [1,K-1]$. Indeed, these graphs provide a canonical way of considering coverings in which \emph{both} past and future values of the discrete state are considered, i.e. providing \emph{mixed} memory/future conditions. We show in the next section that, even in the case of switched linear systems, this approach can improve the numerical performance  with respect to the classical De Bruijn-based conditions.
\end{rmk}

\section{Numerical Examples: Stability of Switched Linear Systems}\label{sec:LinearAndNumerical}
In this section we apply our results to some numerical examples, in order to demonstrate that the abstract setting developed here allows to improve the efficiency of the numerical schemes.  We first recall the definition and properties of the (constrained) joint spectral radius of linear switched systems, since it will be used as a measure of the performance of the proposed approaches.
\subsection{The (constrained) joint spectral radius and stability guarantees for linear switched systems}
In this subsection, we consider linear switched systems. In particular, let us consider a finite alphabet $\Sigma$ and a set of matrices $\cA=\{A(i)\}_{i\in \Sigma}\subset \R^{n\times n}$. 
Given a sofic shift $Z$ we define the $Z$-\emph{constrained joint spectral radius} ($Z$-JSR) of $\cA$  by
\begin{equation}
\rho(\cA,Z)=\lim_{k\to \infty}\sup_{\z\in Z}\|S(k,\z)\|^{\frac{1}{k}}
\end{equation}
where, we recall, $S(\cdot,\z)$ denotes the state-transition matrix, see the formal definition in~\eqref{eq:TransitionMatrix}, in Appendix. We refer to~\cite{Jung09}~\cite{PEDJ:16} for equivalent definitions and further discussion.
Intuitively $\rho(\cA,Z)$ characterizes the UES property of the corresponding linear switched system since it is equal to the infimum over the $\gamma$ for which~\eqref{eq:ExponentialStabilityIneq} holds (see \cite{Jung09}), and thus, the system is UES if and only if $\rho(\cA,Z)<1$. It is well-known that, in a general case, computing and approximating the (constrained) joint spectral radius is numerically demanding, see~\cite{TsiBlo97}.
On the other hand, sequence-dependent Lyapunov functions represent a tunable and handy tool to provide upper bounds for $\rho(\cA,Z)$. More precisely,  let us consider a $\g$-covering $\cC=\{C_1,\dots, C_K\}$, a corresponding graph presentation $\cG=(S,E)$, a family of candidate Lyapunov functions $\cV\subset \cH_+(\R^n,\R)$ (a.k.a. the \emph{template}) where
\begin{equation}
 \cH_+(\R^n,\R)=\left\{v:\R^n\to \R\;\Big\vert \begin{aligned}\;v&\text{ continuous, positive definite}\\ &\text{and positively homogeneous}\footnote{A function $f:\R^n\to \R$ is positively homogeneous if $f(ax)=af(x)$, for any $a\in \R_+$ and any $x\in \R^n$.}\end{aligned}\right\}.
\end{equation} 
 Then, we define $\rho_{\cG,\cV}(\cA,Z)$ as the infimum over the $\rho\in \R_+$ for which functions in $\cV$ satisfying~conditions in Definition~\ref{defn:FiniteGraphLyapunov} exist.
For more related discussion, we refer to~\cite{AJPR:14,DebDel22}. It is important to note that, for any $\cG$ and any $\cV$, we have
\[
\rho_{\cG,\cV}(\cA,Z)\geq \rho(\cA,Z),
\]
i.e. any graph-based Lyapunov function, in the context of switched linear systems, provides an upper bound on the $Z$-JSR. 
In this formalism, Theorem~\ref{lemma:FiniteCoveringLemma}  can be compactly re-stated as follows:
Given any sofic shift $Z\subset \Sigma^\Z$,  any family of matrices $\cA=\{A(i)\}_{i\in \Sigma}\subset \R^{n\times n}$ and any $\cG$, graph presentation of $Z$, it holds that
\begin{equation}\label{eq:cJSRasINF}
\rho(\cA,Z)=\rho_{\cG,\cH_+(\R^n,\R)}(\cA,Z).
\end{equation}

The fact that the (constrained)-joint spectral radius can be obtained as the infimum over the decay rate of Lyapunov functions as in~\eqref{eq:cJSRasINF} is classical, and, at least for the unconstrained case, dates back to the seminal result of~\cite{Rota1960ANO}. For further discussion see, for example,~\cite[Section 1.2.2]{Jung09} and~\cite[Section 2.1]{PEDJ:16}. 

In the following subsection we use particular graph presentations and templates in order to test and verify how the estimation of the $Z$-JSR improves for different graph-induced coverings (and the corresponding graph presentations). 
\subsection{Numerical Examples}
In the papers~\cite{DelRosJun23b,DelRosJun23a} a comparison of Lyapunov conditions is provided for different memory and future coverings, considering the class of quadratic candidate Lyapunov functions and in the arbitrary switching case (i.e. considering only the full shift $\Sigma^\Z$). It is shown that considering uniform (past or future) horizons of observation (i.e., a constant length of the cylinders composing the coverings) does not always provides the best estimation of the JSR. 

In this subsection, instead, applying the general framework and results introduced in previous sections, we compare general finite covering-based conditions, for more general \emph{sofic shifts}. Moreover, we consider templates of candidate Lyapunov functions that differ from the classical case of \emph{quadratic} functions, to show the generality of the proposed techniques.
\begin{example}\label{ex:positiveExample}{\emph{(Positive system evolving on the golden-mean shift)}}\\
Consider the non-negative matrices
\[
A(a):=\begin{bmatrix} 0.2 & \,0.1& \,0\\ 0.6 & \,0.6& \,0.5\\0.6 & \,0.3& \,0.2\end{bmatrix},\;\;\; A(b):=\begin{bmatrix} 0.1 & \,0.2& \,0.3\\ 0.2 & \,0.1& \,0.5\\0.1 & \,0.6& \,0.7\end{bmatrix},
\]
define $f(x,i)=A(i)x$ for any $i\in \Sigma=\{a,b\}$, and consider the golden mean shift $Z_{\text{gm}}$, defined in Example~\ref{ex:GoldenMean}.
We want to compare the stability conditions arising from the $5$ different presentations of $Z_{\text{gm}}$ provided by the first De Bruijn graphs depicted in Figure~\ref{Fig:Ex1DBgraph}. Moreover, since $A_a$ and $A_b$ are non-negative, we can consider the template of \emph{co-positive linear norms} defined as follows. Given a positive vector $v\in \R^n$, $v>_c0$ (here and in what follows $>_c$ denotes the component-wise inequality sign), we define $p_v:\R^n\to \R$ by $p_v(x)=v^\top |x|$ where  we define $|x|:=(|x_1|,\dots,|x_n|)^\top\in \R^n_{\geq 0}$. The function $p_v$ is refereed to as the \emph{primal copositive linear norm} associated to $v$. Given an $n\in \N$, we thus defined the template of primal copositive linear norms by
\[
\cP=\{p_v:\R^n\to \R\;\vert\;v\in \R^n,\;v>_c0\}.
\]
In our context, primal copositive linear norms provide a convenient template of candidate Lyapunov functions, since it can be verified that, for any non-negative matrix $A\in \R^{n\times n}_{\geq 0}$, we have
        \[
p_{v_2}(Ax)\leq p_{v_1}(x),\;\;\forall x\in \R^n\;\;\Leftrightarrow\;\; A^\top v_2-v_1\leq_c 0.
        \]
 Thus, in this context, the feasibility of Lyapunov inequalities encoded in a graph, as in Definition~\ref{defn:FiniteGraphLyapunov}, can be checked via \emph{Linear Programming}. For more details concerning this family of functions and its use in Lyapunov theory, we refer to~\cite{ForVal12,DebDel22} and references therein.
We report in Table~\ref{Table:FirstTable} the corresponding upper bounds provided by computing $\rho_{\cG,\cP}(\cA,Z_{gm})$ with respect to the $5$ graphs in Figure~\ref{Fig:Ex1DBgraph}.
\begin{table}[b!]
\vspace{0.1cm}
\centering
\caption{\normalfont {\footnotesize Numerical upper bounds for~Example~\ref{ex:positiveExample}, obtained with De Bruijn graphs in Figure~\ref{Fig:Ex1DBgraph}.}}
\begin{tabular}{ |c|c|c||c|c|c|} 
 \hline
Graph: & $\cG_{0,1}(Z_{\text{gm}})$ & $\cG_{1,1}(Z_{\text{gm}})$ & $\cG_{0,2}(Z_{\text{gm}})$  & $\cG_{1,2}(Z_{\text{gm}})$&$\cG_{2,2}(Z_{\text{gm}})$\\ 
 \hline
  $\rho_{\cG,\cP}(\cA,Z_{\text{gm}})$ & $1.2714$ & $1.0993$ & $1.2714$ &$1.0993$ &$1.0944$ \\
  \hline
\end{tabular}\label{Table:FirstTable}
\end{table}
We note that, for this example, it is more efficient from a numerical point of view, to use the De Bruijn graphs related to the ``future'', i.e. $\cG_{1,1}(Z_{\text{gm}})$ and $\cG_{2,2}(Z_{\text{gm}})$. Indeed, note that the dimension of the linear program induced by $\cG_{1,1}(Z_{\text{gm}})$ (resp. $\cG_{2,2}(Z_{\text{gm}})$) is equal, in terms of number of variables and inequalities, to the one induced by $\cG_{0,1}(Z_{\text{gm}})$ (resp. $\cG_{0,2}(Z_{\text{gm}})$ and $\cG_{1,2}(Z_{\text{gm}})$). On the other hand the optimal value (i.e. the value of $\rho_{\cG,\cP}(\cA,Z_{\text{gm}}))$) obtained by considering  $\cG_{1,1}(Z_{\text{gm}})$ (resp. $\cG_{2,2}(Z_{\text{gm}})$) is strictly smaller than the one provided by $\cG_{0,1}(Z_{\text{gm}})$ (resp. $\cG_{0,2}(Z_{\text{gm}})$ and $\cG_{1,2}(Z_{\text{gm}})$).
\end{example}
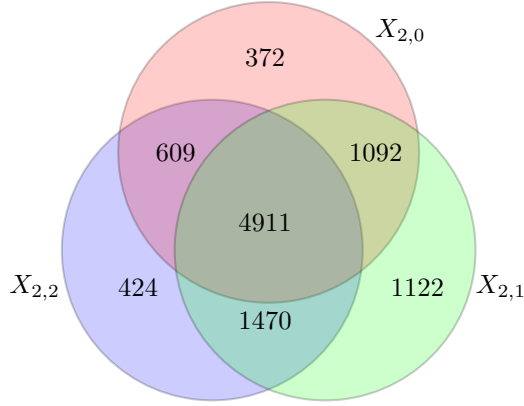
\begin{figure}[t!] 
  \centering
  \vspace{0.0cm}

\begin{tikzpicture}

 \tikzset{venn circle/.style={draw,circle,thick,minimum width=4cm,fill=#1,opacity=0.2}}
  \node [venn circle = blue] (A) at (0,0) {};
  \node [venn circle = red] (B) at (60:1.5cm) {};
  \node [venn circle = green] (C) at (0:1.5cm) {};
\node[left] at (-0.6,-0.5) {$424$}; 
\node[left] at (-1.9,-0.5) {$X_{2,2}$}; 
  \node[left] at (-0.1,1.3 ) {$609$}; 
  \node[left] at (3.2,-0.5) {$1122$};   
    \node[left] at (4.3,-0.5) {$X_{2,1}$};  
  \node[right] at (1.5,1.3 ) {\;\;$1092$};   
  \node[below] at (0.7,0.6){$4911$};
    \node[below] at (0.7,2.8){$372$};
     \node[below] at (2.5,3.2){$X_{2,0}$};
     \node[below] at (0.7,-0.7){$1470$};
\end{tikzpicture} 
\caption{Statistical comparison of the graphs in Figure~\ref{Fig:GeneralDBgraph} performed in Example~\ref{ex:StatCOmparison}. The figure shows, for each subset of the set of considered De Bruijn criteria, the number of examples on which these particular criteria give the best performance}\label{fig:EulerVenn}
  \end{figure}

\begin{example}\label{ex:StatCOmparison}{\emph{(Statistical comparison of De Bruijn graphs of equal order).}}\\
In this example we consider the full shift $\Sigma^\Z$ with $\Sigma=\{a,b\}$ and we compare the performance of the three
De Bruijn graphs of order $2$ depicted in Figure~\ref{Fig:GeneralDBgraph}. In this case we consider the template of \emph{quadratic-diagonal norms} defined as follows: given any positive vector $v\in \R^n_{>0}$, considering the positive definite diagonal matrix $P_v=\text{diag}(v)\in \R^{n\times n}$, we define $q_v:\R^n \to \R$ by $q_v(x):=\sqrt{x^\top P_vx}$.
\[
\cD:=\{q_v:\R^n\to \R\;\vert\;v\in \R^n_{>0}\}.
\]

To perform the comparison, we sample $10000$ random couples of matrices $\cA=\{A_a,A_b\}\subset \R^{2\times 2}$, with each element uniformly sampled in the interval $[-10,10]$. For each sampled couple of matrices $\cA_\omega$, we compute $\rho_{\cG,\cD}(\cA_\omega,\Sigma^\Z)$ for each one of the De Bruijn graphs $\cG_{2,t}$ for $t\in \{0,1,2\}$, solving the resulting optimization problem. The numerical results are reported in the Euler-Venn diagram in Figure~\ref{fig:EulerVenn}. The diagram reads as follows: the numerical value in the set $\cG_{2,t}\setminus(\cG_{2,t'}\cup\cG_{2,t''})$ (for pairwise distinct $t,t'$ and $t''$) corresponds to the number of systems for which the graph $\cG_{2,t}$ provides a strictly better estimation of the JSR with respect to $\cG_{2,t'}$ and $\cG_{2,t''}$, (i.e. the number of instances for which $\rho_{\cG_{2,t},\cD}(\cA_\omega,\Sigma^\Z)<\rho_{\cG_{2,t'},\cD}(\cA_\omega,\Sigma^\Z)$ and $\rho_{\cG_{2,t},\cD}(\cA_\omega,\Sigma^\Z)<\rho_{\cG_{2,t''},\cD}(\cA_\omega,\Sigma^\Z)$). Similarly the value in the intersections corresponds to instances in which some estimations coincide: for example the intersection $X_{2,0}\cap X_{2,1}\cap X_{2,2}$ corresponds to the case in which the obtained upper bound on $\rho(\cA_\omega, \Sigma^\Z)$ coincides, for the three graphs.  From this analysis, despite that in almost half of the cases the three graphs/coverings lead to the same upper bounds, statistically the graph $\cG_{2,1}$ performs better with respect to $\cG_{2,0}$ and $\cG_{2,2}$. Recalling Definition~\ref{defn:GenDeBrunjii}, this implies that, at least for this class of systems and for the chosen templates, conditions arising from \emph{mixed future/memory} coverings, provide more appealing stability conditions with respect to purely memory- and future-based conditions. Indeed, recall from~Proposition~\ref{prop:propDeBrunji}, that the three graphs $\cG_{2,0}$, $\cG_{2,1}$ $\cG_{2,2}$ represent graph-induced coverings defined by cylinders of length $2$, but interpreted differently, in term of possible past and future values of switching sequences. Rephrasing,  the graph $\cG_{2,1}$, which intuitively encodes the idea of ``looking'' one step back and one ahead, performs better that the one looking $2$ steps back ($\cG_{2,0}$, inducing a memory covering) and the one looking $2$ steps ahead ($\cG_{2,2}$, inducing a future covering). We have thus numerically justified the benefits of the general framework introduced in this paper, with respect to classical conditions based on memory/future coverings, as in~\cite{LeeDull06,Thesis:Essick,EssickLee14,DelRosJun23b,DelRosJun23a}. 
\end{example}

\section{Conclusions}\label{Sec:Conclu}
We provided a new general model of dynamical systems evolving jointly on a continuous state space and on a discrete space given by sofic shifts. This model encapsulates many existing ones, such as switched systems and time-varying systems.

We provided a complete Lyapunov characterization of the uniform stability, and sufficient Lyapunov conditions based on coverings of sofic shifts induced by graphs. This framework allowed us to generalize, unify and relate existing literature studying graph-based/path-complete Lyapunov functions from one side and, on the other side, results involving exploitation of past/future information on the discrete state. 

We showed, in the particular case of linear switched systems, that our theory can lead to improved algorithms for the stability analysis: not only it provides an interpretation of multiple Lyapunov functions in terms of covering the switching signals, but it naturally leads to new stability conditions, that outperform previous ones, at equal computational cost.

In future research, we will push further the application of this framework in order to better understand how Lyapunov criteria compare with each other, thanks to the covering interpretation. This could lead to tailored Lyapunov criteria, that would be optimized thanks to their language-theoretic interpretation. We also plan to investigate the application of our theory to more general hybrid systems.
\appendix
\section{Technical Proofs}
In this appendix we collect some technical proofs.
\subsection{Proof of Theorem~\ref{lemma:ConverseNonLinear}}\label{appendix:ProofConverseNonLinear}
\begin{proof}
Let us prove the \emph{``if''} part first.
Suppose there exists a sd-LF  for system~\eqref{eq:System}, consider any $x\in  \R^n$ and any $\z\in Z$. Computing, using~\eqref{eq:Decreasing1} and proceeding by induction recalling~\eqref{eq:SemiGroupProperty}, we have
\[
\begin{aligned}
\cV(\Phi(k,x_0,\z), \sigma^k(\z))&=\cV(f(\Phi(k-1,x_0,\z), \sigma^{k-1}(\z)),\, \sigma^k(\z))\\& \leq \gamma \cV(\Phi(k-1,x_0,\z),\sigma^{k-1}(z))\leq \dots \leq  \gamma^k \cV(x,\z).
\end{aligned}
\]
Now using~\eqref{Eq:Sandwich1}, we have
\[
|\Phi(k,x_0,\z)|\leq \alpha_1^{-1}(\gamma^k\alpha_2(|x|)),\;\;\;\forall \;k\in \N,
\]
concluding the proof since the function $\wt \beta(s,k):=\alpha_1^{-1}(\gamma^k \alpha_2(s))$ is of class $\cKL$.

For the \emph{``only if''} part, the result is inspired by the techniques introduced in~\cite{JiaWan02,KellTeel}. Suppose system~\eqref{eq:System} is GUAS, and consider the function~$\beta\in \cKL$ introduced in Definition~\ref{defn:GUAS}. Using \cite[Proposition 7]{Son98}, for any $\gamma\in (0,1)$, there  exist functions $\rho_1,\rho_2\in \cK_\infty$ such that
\[
\beta(s,k)\leq \rho_1(\gamma^k\rho_2(s)),\;\;\forall s\in \R_{\geq 0},\;\;\forall\;k\in \N, 
\]
and thus we have
\begin{equation}\label{eq:InequalityConverseLemma}
|\Phi(k,x,\z)|\leq \rho_1(\gamma^k\rho_2(|x|)), \forall k\in \N,\;\forall x\in \R^n,\;\forall \,\z\in Z.
\end{equation}
Based on~\eqref{eq:InequalityConverseLemma}, we now construct a sd-LF for system~\eqref{eq:System}. For completeness, we actually propose two different constructions, memory- and future- based, respectively.
\\
\noindent
\emph{Future-Based:}
Let us define, for any $x\in \R^n$ and any $\z\in Z$,
\begin{equation}\label{eq:DefnCOnverseNonLinear}
\cV_+(x,\z):=\sup_{k\in \N}\left \{\frac{1}{\gamma^k}\rho_1^{-1}(|\Phi(k,x,\z)|)\;\right \}.
\end{equation}
First of all using~\eqref{eq:InequalityConverseLemma}, it can be seen that
\[
\rho_1^{-1}(|x|)\leq\cV_+(x,\z)\leq \rho_2(|x|),\;\;\;\forall \;x\in \R^n,\;\forall\;\z\in Z.
\]
For any $x\in \R^n$ and any $\z\in \Sigma^\Z$ and any $k\in \N$, using~\eqref{eq:SemiGroupProperty}, we have
\[
\frac{1}{\gamma^{k}}\rho_1^{-1}(|\Phi(k,f(x,\z),\,\sigma(\z))|)=\gamma \frac{1}{\gamma^{k+1}}\rho_1^{-1}(|\Phi(k+1,\,x,\,\z)|).
\]
Thus, computing
\[
\begin{aligned}
\cV_+(f(x,\z), \sigma(\z))&=\sup_{k\in \N}\left \{\frac{1}{\gamma^{k}}\rho_1^{-1}(|\Phi(k,f(x,\z),\sigma(\z))|)\;\right \}\\
&=\gamma \sup_{k\in \N\setminus\{0\}}\left \{\frac{1}{\gamma^{k}}\rho_1^{-1}(|\Phi(k,\,x,\,\z)|))\;\right \}\leq \gamma\cV_+(x,\z),
\end{aligned}
\]
concluding the proof.
\\
\noindent
\emph{Memory-based:}
Let us note that~\eqref{eq:InequalityConverseLemma} using~\eqref{eq:SemiGroupProperty} in particular implies
\begin{equation}\label{eq:InequalityConverseLemma2}
|x|\leq \rho_1(\gamma^{-k}\rho_2(|y|)),\;\; \forall k\in \Z_-,\;\forall x\in \R^n,\;\forall \,\z\in Z, \;\forall y\in \Phi(k,x,\z).
\end{equation}
We thus define 
\begin{equation}\label{eq:DefnCOnverseNonLinearPAST}
\cV_-(x,\z):=\inf_{k\in \Z_{-}\cup\{0\}}\inf_{y\in\Phi(k,x,\z)}\left \{\gamma^{-k}\rho_2(|y|)\;\right \}.
\end{equation}
By~\eqref{eq:InequalityConverseLemma2} we have
\[
\rho_1^{-1}(|x|)\leq\cV_-(x,\z)\leq \rho_2(|x|),\;\;\;\forall \;x\in \R^n,\;\forall\;\z\in Z.
\]
Now, consider any $\z\in Z$, any $x\in \R^n$, any $k\in \Z_-\cup\{0\}$ and any $y\in \Phi(k,x,\z)$; we note that this implies $y\in \Phi(k-1, f(x,\z),\sigma(\z))$. We thus have
\[
\begin{aligned}
\cV_-(x,\z)&=\inf_{k\in \Z_{-}\cup\{0\}}\inf_{y\in\Phi(k,x,\z)}\left \{\gamma^{-k}\rho_2(|y|)\;\right \}\geq \\&\inf_{k\in \Z_{-}}\inf_{y\in\Phi(k,f(x,\z),\sigma(\z))}\left \{\gamma^{-k-1}\rho_2(|y|)\;\right \}\geq\frac{1}{\gamma}\cV_-(f(x,\z),\sigma(\z)) ,
\end{aligned}
\]
proving~\eqref{eq:Decreasing1}.
\end{proof}
We note that we did not prove any continuity property of functions $\cV_{+}(\cdot,\z)\to \R$ and $\cV_-(\cdot,\z)\to \R$ on $\R^n\setminus \{0\}$. Indeed, the existence of continuous Lyapunov functions (in $\R^n\setminus \{0\}$) can be stated, applying  a \emph{smoothing technique} to the functions $\cV_{+}(\cdot,\z)$ and  or $\cV_-(\cdot,\z)$ constructed in proof of Theorem~\ref{lemma:ConverseNonLinear}. This technical topic is not reported here, we refer~to~\cite{KellTeel} for the details.

\subsection{Proof of Theorem~\ref{Lemma:ConverseLinear}}\label{Appendix:ProofconverseLinear}
\begin{proof}
For the implication $(2)\,\Rightarrow\,(1)$, for any $\wt \gamma\in (\gamma,1)$ consider the function $\cV:\R^n\times Z\to \R$ defined by $\cV(x,\z):=\sqrt{x^\top Q(\z)x}$, it is easy to see that with this definition we have that $\cV$ satisfies~\eqref{Eq:Sandwich1} and~\eqref{eq:Decreasing1} and that $\cV(\cdot, \z):\R^n\to \R$ is a norm, for all $\z \in Z$. Thus, as in proof of Theorem~\ref{lemma:ConverseNonLinear}, if such $Q:Z\to \mathbb{S}^n_+$ exists, $\forall \z\in Z,\;\forall k\in \N$, $\forall\;x\in \R^n$, we have
\[
|\Phi(k,x,\z)|\leq \frac{1}{M_1}\cV(\Phi(k,x,\z),\sigma^k(z))\leq \frac{1}{M_1}\wt\gamma^k\, \cV(x,\z)\leq \frac{M_2}{M_1}\wt\gamma^k |x|,\;\;\;
\]
thus proving UES with decay rate $\wt \gamma$. By arbitrariness of $\wt \gamma\in (\gamma,1)$ we conclude.

For the implication $(1)\,\Rightarrow\,(2)$, consider any $\wt \gamma\in (\gamma,1)$ and let us fix any $\wt \gamma_1\in (\gamma,\wt \gamma)$. By assumption, system~\eqref{eq:System} is UES with decay rate $\wt \gamma_1$. First, we introduce the \emph{state-transition matrix}, denoted by
\begin{equation}\label{eq:TransitionMatrix}
S(k,\z):=A(\sigma^{k-1}(\z))\cdot\cdot\cdot A(\z),\;\;\;\forall \;\z\in \Sigma^\Z,\forall \;k\in \N,
\end{equation}
with the convention $S(0,\z)=I_n$, for all $\z\in Z$.
We thus have $\Phi(k,x_0,\z)=S(k,\z)x_0$, and  by~\eqref{eq:ExponentialStabilityIneq}, we have
\begin{equation}\label{eq:BoundOnS}
|S(k,\z)|\leq M \wt\gamma_1^k,\;\;\forall \z\in Z,\;\forall k\in \N,
\end{equation}
where $|\cdot|$ denotes here the induced norm on the matrices space. 
We provide, as in the nonlinear case treated in Theorem~\ref{lemma:ConverseNonLinear}, both a memory- and a future- based construction. 
\\
\noindent
\emph{Future-based:}
Let us define
\begin{equation}\label{eq:ConverseDefinitionLinear}
Q_+(\z)=\sum_{k=0}^\infty \frac{1}{\wt  \gamma^{2k}}S(k,\z)^\top S(k,\z), \;\;\;\;\forall\;\z\in Z.
\end{equation}
It is clear that $Q_+(\z)\succeq I_n$ for all $z\in Z$ (since $I_n$ is the first term of the sum), and then, by~\eqref{eq:BoundOnS} we have
\[
Q_+(\z)\preceq M^2 \sum_{k=0}^\infty \left(\frac{\wt \gamma_1}{\wt \gamma }\right)^{2k} I_n\preceq M^2\frac{\wt \gamma^2}{\wt \gamma^2- \wt\gamma_1^2} I_n,
\]
proving~\eqref{eq:SandwichLinear}. Now, for every $\z\in Z$, we have
\[
\begin{aligned}
A(\z)^\top Q_+(\sigma(\z))A(\z)&=\sum_{k=0}^\infty \frac{1}{\wt \gamma ^{2k}}A(\z)^\top S(k, \sigma(\z))^\top S(k, \sigma(\z))A(\z)\\&=\sum_{k=1}^\infty \frac{1}{\wt \gamma^{2(k-1)}} S(k,\z)^\top S(k,\z) =\wt \gamma^2(\sum_{k=1}^\infty \frac{1}{\wt \gamma^{2k}} S(k,\z)^\top S(k,\z))\\&=\wt \gamma^2 Q_+(\z)-\wt\gamma^2 I_n\prec\wt \gamma^2 Q_+(\z), 
\end{aligned}
\]
proving~\eqref{eq:PseudoRiccatti}.
\\
\noindent
\emph{Memory-Based:}
First of all we note that, using the Schur complement (see \cite[Appendix A.5.5.]{BoydVan}), if~\eqref{eq:SandwichLinear} holds, condition~\eqref{eq:PseudoRiccatti} is equivalent to 
\begin{equation}\label{eq:ShurCompelment}
\wt \gamma^2 Q(\sigma(\z))^{-1}-A(\z)Q(\z)^{-1}A(\z)^\top \succ 0.
\end{equation}
Let us define
\begin{equation}\label{eq:ConverseDefinitionLinear2}
P(\z):=\sum_{k=0}^\infty \frac{1}{\wt \gamma^{2k}}S(k,\sigma^{-k}(\z))S(k,\sigma^{-k}(\z))^\top.
\end{equation}
Bounds as in~\eqref{eq:SandwichLinear} for $P(\z)$ can be proven, as in the future-based case, using~\eqref{eq:BoundOnS}.
Computing,  we have
\[
\begin{aligned}
A(\z)P(\z)A(\z)^\top&=\sum_{k=0}^\infty \frac{1}{\wt \gamma^{2k}}A(\z)S(k,\sigma^{-k}(\z))S(k,\sigma^{-k}(\z))^\top A(\z)^\top\\&=\sum_{k=0}^\infty \frac{1}{\wt \gamma^{2k}}S(k+1,\sigma^{-k}(\z))S(k+1,\sigma^{-k}(\z))^\top\\&=\wt \gamma^2\sum_{k=1}^\infty \frac{1}{\wt \gamma^{2k}}S(k,\sigma^{-k+1}(\z))S(k,\sigma^{-k+1}(\z))^\top=\wt \gamma^2(P(\sigma(\z))-I_n)
\end{aligned}
\]
and thus 
\[
\wt \gamma^2P(\sigma(\z))- A(\z)P(\z)A(\z)^\top=\wt \gamma^2 I_n\succ 0,
\]
defining $Q_-(\z)=P(\z)^{-1}$, recalling~\eqref{eq:ShurCompelment}, and the fact that $aI_n\prec R\prec bI_n$ if and only if $b^{-1}I_n\prec R^{-1}\prec a^{-1}I_n$, for any $a\leq b$ and any $R=R^\top\in \R^{n\times n}$, we conclude.
\end{proof}
\subsection{Proof of Theorem~\ref{lemma:FiniteCoveringLemma}}\label{appendix:ProofFinite}
\begin{proof}
The ``if'' part trivially follows from Proposition~\ref{prop:FromFiniteToInfinite}. We prove the ``only if'' part, and only in the ``future based'' version.
As in the proof of Theorem~\ref{lemma:ConverseNonLinear}, the GUAS property is equivalent to the existence of a $\gamma\in [0,1)$ and 
 functions $\rho_1,\rho_2\in \cK_\infty$ such that
\[
|\Phi(k,x_0,\z)|\leq \rho_1(\gamma^k\rho_2(|x_0|)),\;\;\forall x_0\in \R^n, \forall \z\,\in Z,\;\;\forall\;k\in \N.
\]
Let us define, for every $C\in \cC$,
\[
W_+(x,C)=\sup_{\substack{\z\in C\\ k\in \N}}\left \{\frac{1}{\gamma^k}\rho_1^{-1}(|\Phi(k,x,\z)|)\;\right \}.
\]
We trivially have $\rho_1(|x|)\leq W_+(x,C)\leq \rho_2(|x|)$, for all $C\in \cC$ and all $x\in \R^n$. Now consider $C,D\in \cC$ and $i\in \Sigma$ and any $x\in \R^n$. We note that by~\eqref{eq:SuffPropagation1}, for any $\z\in D$ we have that exists $\w\in C$ such that $i \cdot \z^+= \w^+$. Thus computing
\[
\begin{aligned}
W_+(f(x,i),D)&=\sup_{\substack{\z\in D\\ k\in \N}}\left \{\frac{1}{\gamma^k}\rho_1^{-1}(|\Phi(k,f(x,i),\z)|)\;\right \}\leq\sup_{\substack{\w\in C\\ k\in \N}}\left \{\frac{1}{\gamma^k}\rho_1^{-1}(|\Phi(k+1,x,\w)|)\;\right \}\\&=\gamma\sup_{\substack{\w\in C\\ k\in \N\setminus\{0\}}}\left \{\frac{1}{\gamma^k}\rho_1^{-1}(|\Phi(k,x,\w)|)\;\right \}\leq \gamma W_+(x,C),
\end{aligned}
\]
concluding the proof. For a backward construction using~\eqref{eq:SuffPropagation2}, the reasoning is similar to the one in the proof of Theorem~\ref{lemma:ConverseNonLinear}, and it is not reported here.
\end{proof}

\bibliography{biblio} 
\bibliographystyle{plain}

\end{document}